\documentclass[12pt,twoside,reqno,openany]{amsart}
\usepackage{fullpage,times,url,amssymb,amsbsy,amsmath,amsthm,
amsfonts,amssymb,amscd,stmaryrd,graphics,color,xypic,footmisc,xy,
fancyhdr,multicol,fancybox,graphicx,mathrsfs,rotating,ifthen,
wasysym,shuffle,txfonts}
\usepackage{pifont}
\usepackage{centernot}
\usepackage{amsthm}
\usepackage{mathrsfs}
\usepackage{color}
\usepackage{extarrows}
\usepackage[colorlinks,linkcolor=red,anchorcolor=green,
citecolor=blue]{hyperref}
\usepackage{eurosym}
\usepackage{mathrsfs}
\usepackage{multirow}
\usepackage{arydshln}
\usepackage[T1]{fontenc}

\renewcommand{\dim}{\text{\footnotesize\sf dim}}

\renewcommand{\max}{\text{\footnotesize\sf max}}
\newcommand{\cdim}{\text{\footnotesize\sf codim}}
\newcommand{\rank}{\text{\footnotesize\sf rank}}

\newcommand{\vf}{\vfill\end{document}}
\newcommand{\explain}[1]{\text{\scriptsize\sf [#1]}}

\newcommand{\dformal}{{\text{\usefont{T1}{qcs}{m}{sl}d}}}

\let\mathcal\mathscr

\newtheorem{The}{Theorem}[section]

\newtheorem{Theorem}[The]{Theorem}
\newtheorem{Proposition}[The]{Proposition}

\newtheorem{Observation}[The]{Observation}

\theoremstyle{definition}

\newtheorem{Definition}[The]{Definition}

\newtheorem{Remark}[The]{Remark}

\makeatletter
\def\noqed{\let\QED@stack\@empty}
\makeatother

\definecolor{black}{cmyk}{1.,1.,1.,1.0
}
\definecolor{blue}{cmyk}{1.,1.,0.,0.63}
\definecolor{red}{cmyk}{0.,1.,1.,0.63}
\definecolor{green}{cmyk}{1.,0.,1.,0.63}

\begin{document}

\title{
Generalized Brotbek's 
\\
symmetric differential forms
\\ 
and applications
}

\author{Song-Yan Xie}

\thanks{This work was supported by the {\sl Fondation Math\'ematique Jacques Hadamard} through the grant  N\textsuperscript{o} ANR-10-CAMP-0151-02 within the ``Programme des Investissements d'Avenir''.}

\address{Laboratoire de Math\'ematiques d'Orsay,
 Univ. Paris-Sud, 
CNRS, Universit\'e Paris-Saclay, 
91405 Orsay, France.
 }

\email{songyan.xie@math.u-psud.fr}



\maketitle

\begin{abstract}
Over an algebraically closed field $\mathbb{K}$
with any characteristic, on
an $N$-dimensional smooth projective $\mathbb{K}$-variety $\mathbf{P}$ equipped with $c\geqslant N/2$ very ample line bundles $\mathcal{L}_1,\dots,\mathcal{L}_c$,
we study the {\sl General Debarre Ampleness Conjecture}, which expects that for all large degrees
$d_1,\dots,d_c\geqslant \texttt{d}\gg 1$,
for generic $c$ hypersurfaces
$
H_1\in \big|\mathcal{L}_1^{\,\otimes\,d_1}\big|$, $\dots$, 
$H_c\in \big|\mathcal{L}_c^{\,\otimes\,d_c}\big|$,
the complete intersection 
$X:=H_1 \cap \cdots \cap H_c$ 
 has
ample cotangent bundle $\Omega_X$.

First, we introduce a notion of {\sl formal matrices}
and a {\sl dividing device}
to produce negatively twisted symmetric differential forms, which extend the previous constructions of
Brotbek and the author.
Next, we adapt the moving coefficients method (MCM),
and we establish that, if $\mathcal{L}_1,\dots,\mathcal{L}_c$ are {\sl almost proportional} to each other, then the above conjecture holds true. 
Our method is effective:
for instance, in the simple case $\mathcal{L}_1=\cdots=\mathcal{L}_c$, we provide an explicit lower degree bound $\texttt{d}=N^{N^2}$. 
\end{abstract}

\section{\bf Introduction}
Smooth projective varieties having ample cotangent bundle 
suit well with the phenomenon\big/phil\-osophy that `{\em geometry governs arithmetic}', in the sense that,
on one hand, over the complex number field $\mathbb{C}$,
none of them contain any entire curve, on the other hand, over a number field $K$, each of them is expected to possess only finitely many $K$-rational points (Lang's conjecture). For instance in the one-dimensional case,
the first property is due to the Uniformization Theorem
and the Liouville's Theorem, while the second assertion
is the famous Mordell Conjecture\big/Faltings's Theorem.

For a long time, 
few such varieties were known, even though
they were expected to be reasonably abundant.
In this aspect, Debarre 
conjectured
in~\cite{Debarre-2005}
that the intersection of $c\geqslant N/2$ generic hypersurfaces of large degrees in $\mathbb{P}_{\mathbb{C}}^N$ should have ample cotangent bundle.

By introducing the moving coefficients method (MCM)
and the product coup,
the Debarre Ampleness Conjecture was first established in~\cite{Xie-2015-arxiv}, with an additional effective lower degree bound.

\begin{Theorem}[\cite{Xie-2015-arxiv}]
\label{Debarre Ampleness Conjecture}
 The cotangent bundle 
$\Omega_X$ 
 of the complete intersection 
$X:=H_1 \cap \cdots \cap H_c 
\subset\mathbb{ 
P}_{\mathbb{C}}^N$
 of $c\geqslant N/2$ generic hypersurfaces 
$H_1,\dots,H_c$ 
 with degrees $d_1,\dots,d_c\geqslant N^{N^2}$ is ample.
\end{Theorem}

The proof there extends the approach of~\cite{Brotbek-2014-arxiv}, by adding
four major ingredients as follows.

\begin{itemize}

\smallskip\item[\bf {(1)}]\,
Generalizations of Brotbek's symmetric differential forms~\cite[Lemma~4.5]{Brotbek-2014-arxiv}
by means of a geometric approach, and also by a
scheme-theoretic approach.

\smallskip\item[{\bf(2)}]\,
Make use of `hidden' symmetric differential forms constructed over any intersection of 
Fermat-type hypersurfaces with coordinate hyperplanes:
\[
\underbrace{H_1 \cap \cdots \cap H_c}
_{=\,X}\, 
\cap\,
\{z_{\nu_1}=\cdots=z_{\nu_\eta}=0\}
\qquad
{\scriptstyle(\forall\,
\eta\,=\,1\,\cdots\,N-c-1;\,\,
0\,\leqslant\,\nu_1\,<\,\cdots\,<\,\nu_\eta\,\leqslant\,N)}.
\]

\smallskip\item[{\bf(3)}]\,
`Flexible' hypersurfaces designed by MCM, 
which produce many more negatively twisted symmetric differential forms than pure Fermat-type ones.

\smallskip\item[{\bf(4)}]\,
The {\sl product coup}, which produces ample examples of all large degrees $d_1,\dots,d_c$.

\end{itemize}

\medskip
Recently, Brotbek and Darondeau~\cite{Brotbek-Lionel-2015} provided another approach to  the Debarre Ampleness Conjecture, by
means of
new constructions and
deep theorems in algebraic geometry.
As mentioned in~\cite[p.~2]{Brotbek-Lionel-2015},
 it is tempting to extend the Debarre Ampleness Conjecture from projective spaces to projective varieties, equipped with several very ample line bundles.   

\medskip
\noindent
{\bf General Debarre Ampleness Conjecture.}\,
For any smooth projective $\mathbb{K}$-variety $\mathbf{P}$ of dimension
$
N
\geqslant
1
$,
for any positive integer
$c\geqslant N/2$,
for any very ample line bundles $\mathcal{L}_1,\dots,\mathcal{L}_{c}$ over $\mathbf{P}$,
there exists some lower bound:
\[
\texttt{d}\,
=\,
\texttt{d}\,(\mathbf{P}, \mathcal{L}_1,\dots,\mathcal{L}_{c})\
\gg\
1
\] 
such that,
for all large degrees
$
d_1,
\dots,
d_c
\geqslant
\texttt{d}
$,
for $c$ generic
 hypersurfaces: 
\[H_1\in \big{|}\mathcal{L}_1^{\,\otimes\,d_1}\big|, \dots, 
H_c\in \big|\mathcal{L}_c^{\,\otimes\,d_c}\big|,
\]
the complete intersection 
$X:=H_1 \cap \cdots \cap H_c$ 
 has
ample cotangent bundle $\Omega_X$.

\medskip
Sharing the same flavor as~\cite[p.~6, Conjecture 1.5]{Xie-2015-arxiv},
this general conjecture attracts our interest.
To this aim, we develop further our previous method in~\cite{Xie-2015-arxiv}, and generalize
several results.

We work over an algebraically closed field $\mathbb{K}$
with any characteristic. 
First of all,
by adapting the techniques in~\cite{Xie-2015-arxiv}, we can confirm the General Debarre Ampleness Conjecture
in the case $\mathcal{L}_1=\cdots=\mathcal{L}_c=\mathcal{L}$.

\begin{Theorem}
\label{Gentle generalization of Debarre conjecture}
Let $\mathbf{P}$ be an $N$-dimensional smooth projective $\mathbb{K}$-variety, equipped with a very ample line bundle $\mathcal{L}$.
For any positive integer
$c\geqslant N/2$,
for all large degrees
$
d_1,
\dots,
d_c
\geqslant
N^{N^2}
$,
for $c$ generic
 hypersurfaces 
$H_1\in \big|\mathcal{L}_1^{\,\otimes\,d_1}\big|, \dots, 
H_c\in \big|\mathcal{L}_c^{\,\otimes\,d_c}\big|
$,
the complete intersection 
$X:=H_1 \cap \cdots \cap H_c$ 
 has
ample cotangent bundle $\Omega_X$.
\end{Theorem}

In fact, we will prove a stronger result, in the case that
the $c$ rays:
\[
\mathbb{R}_{+}\cdot [\mathcal{L}_1],\,
\dots,\,
\mathbb{R}_{+}\cdot [\mathcal{L}_c]\
\subset\
\text{Ample Cone of }\mathbf{P}
\]  
have small pairwise angles.
More rigorously, we introduce the

\begin{Definition}
\label{almost parallel}
Let $\mathbf{P}$ be an $N$-dimensional projective variety, and let $\mathcal{L}$, $\mathcal{S}$ be two ample line bundles on $\mathbf{P}$.
Then $\mathcal{S}$ is said to be {\sl almost proportional} to $\mathcal{L}$, if there exist
two elements $\alpha\in \mathbb{R}_{+}\cdot [\mathcal{S}]$
and $\beta\in \mathbb{R}_{+}\cdot [\mathcal{L}]$
such that $\beta<\alpha<(1+\epsilon_0)\,\beta$,
i.e. both $\alpha-\beta, (1+\epsilon_0)\,\beta-\alpha$ lie in the ample cone of $\mathbf{P}$,
where $\epsilon_0:=3/({N^{N^2/2}}-1)$.
\end{Definition}

The value $\epsilon_0$ is due to the effective degree estimates of MCM, see Proposition~\ref{explain the definition of proportional}.

\begin{Theorem}
\label{Main Theorem, 2nd}
Let $\mathbf{P}$ be an $N$-dimensional smooth projective $\mathbb{K}$-variety, equipped with a very ample line bundle $\mathcal{L}$.
For any integers
$c$, $r\geqslant 0$ with
$2c+r\geqslant N$,
for any $c+r$ ample line bundles $\mathcal{L}_1, \dots, \mathcal{L}_{c+r}$ which are almost
proportional to $\mathcal{L}$, there exists some integer:
\[
\texttt{d}\,
=\,
\texttt{d}\,
(\mathcal{L}_{1},\dots,
\mathcal{L}_{c+r}, \mathcal{L})\
\gg\
1
\]
such that,
for all large integers
$
d_1,
\dots,
d_c,
d_{c+1},
\dots,
d_{c+r}\, 
\geqslant\, 
\texttt{d}
$,
for generic $c+r$
 hypersurfaces: 
\[
H_1
\in
\big|\mathcal{L}_1^{\,\otimes\,d_1}\big|,
\dots,
H_{c+r}\in \big|\mathcal{L}_{c+r}^{\,\otimes\,d_{c+r}}\big|,
\]
the cotangent bundle 
$\Omega_V$ of the intersection of the first $c$ hypersurfaces
$
V
:=
H_1
\cap
\cdots
\cap 
H_c
$
restricted 
to the intersection of all the $c+r$ hypersurfaces
$
X
:=
H_1
\cap
\cdots
\cap 
H_c
\cap
H_{c+1}
\cap
\cdots
\cap 
H_{c+r}
$
is ample.
\end{Theorem}

We will see that in our proof, the lower degree bound $\texttt{d}=\texttt{d}(\mathcal{L}_{1},\dots,
\mathcal{L}_{c+r}, \mathcal{L})$ is effective. In particular, when $r=0$ and all  $\mathcal{L}_1=\cdots=\mathcal{L}_c=\mathcal{L}$ coincide,
we will obtain the effective degree bound $N^{N^2}$ of
Theorem~\ref{Gentle generalization of Debarre conjecture}. See  Subsection~\ref{subsection: proof of theorem 2} for the details.

\medskip
This paper is organized as follows.
In Section~\ref{section: General Strategy},
we outline the general strategy for the Debarre Ampleness Conjecture, which serves as a guiding principle of our approach. Next, in Section~\ref{Section: General Symmetric Differential Forms},
we introduce a notion of {\sl formal matrices},
and use their determinants to produce symmetric differential forms.
Then, for the purpose of making negative twist, we
play a {\sl dividing trick} in Section~\ref{section: A Dividing Trick}, and thus generalize the aforementioned ingredient {\bf (1)}. 
Consequently,
we are able to generalize {\bf (2)} 
in Section~\ref{section: Hidden Symmetric Differential Forms}. Thus, 
by adapting the ingredients {\bf (3)}, {\bf (4)} as well,
we establish Theorem~\ref{Main Theorem, 2nd}
in Section~\ref{Section: Applications of MCM},
by means of the moving coefficients method developed in~\cite{Xie-2015-arxiv}.
Lastly, we fulfill some technical details in Section~\ref{section: Some Technical Details}.

\medskip
It is worth to mention that,
by means of formal matrices,
we can also construct higher order jet differential forms. Therefore, we can also apply MCM
to study the ampleness of certain jet subbundle
of hypersurfaces in $\mathbb{P}_{\mathbb{C}}^N$, notably when $N=3$.
We will discuss this in our coming paper.

\medskip
\noindent {\bf Acknowledgments.}   
I would like to
thank Damian Brotbek
and Lionel Darondeau for inspiring
 discussions.
Also, I thank my thesis advisor     
Jo\"el Merker for valuable suggestions and remarks.

\section{\bf General Strategy}
\label{section: General Strategy}

It seems that, up to date, there has been only one  strategy to settle the Debarre Ampleness Conjecture.
To be precise,
for fixed degrees $d_1,\dots,d_c$ of hypersurfaces, 
the strategy is firstly to choose a
certain subfamily of $c$ hypersurfaces,
and then secondly to
construct sufficiently many 
negatively twisted symmetric differential forms
over the corresponding  subfamily of intersections,
and lastly to narrow their base locus up to discrete points over a generic intersection. Thus, there exists
one desirable ample example in this subfamily, which
suffices to conclude the generic ampleness of the whole family thanks to a theorem of Grothendieck.

Following this central idea, the first result~\cite{Brotbek-2014} was obtained in the case $c=N-2$ for complex surfaces 
$X=H_1\cap\cdots\cap H_c\subset \mathbb{P}_{\mathbb{C}}^N$, by employing a
 method related to Kobayashi hyperbolicity problems, in which the existence\big/quantity of negatively twisted symmetric differential forms
 was guaranteed\big/measured  by the holomorphic Morse inequality.  
Such an approach would fail in the higher dimensional case,
simply because one could not control the base locus of the
 implicitly given symmetric forms.

To find an alternative approach, the key breakthrough happened when Brotbek 
constructed explicit negatively twisted symmetric differential forms
\cite[Lemma~4.5]{Brotbek-2014-arxiv} by a cohomological approach,
for the subfamily of pure Fermat-type hypersurfaces
of the same degree $d+\epsilon$ defined by: 
\[
F_i\,
=\,
\sum_{j=0}^N\,A_i^j\,z_j^{d}
\qquad
{\scriptstyle{(i\,=\,1\cdots\,c)}},
\] 
where $d, \epsilon\geqslant 1$, and where all  coefficients $A_i^j$ are some homogeneous polynomials 
with $\deg\,A_i^j=\epsilon\geqslant 1$. Then,
in the case $4c\geqslant 3N-2$, Brotbek showed that over a generic intersection $X$, the obtained symmetric forms
have discrete base locus, and hence he established the conjectured ampleness. 

However, when $4c<3N-2$, this approach would not work, because the obtained symmetric differential forms
keep positive dimensional base locus, for instance
in the limiting case $2c=N$, there is only one  obtained symmetric form, 
whereas $\dim\, \mathbb{P}(\Omega_X)=N-1 \gg 1$. 

To overcome this difficulty, the author~\cite{Xie-2015-arxiv}
introduced the moving coefficients method (MCM), the 
cornerstone of which is a generalization of Brotbek's symmetric differential forms
for general Fermat-type hypersurfaces defined by:
\begin{equation}
\label{General c+r Fermat type hypersurfaces}
F_i
=
\sum_{j=0}^N\,
A_i^j\,z_j^{\lambda_j}
\qquad
{\scriptstyle{(i\,=\,1\cdots\,c)}},
\end{equation}
where $\lambda_0,\dots,\lambda_N\geqslant 1$
and where all polynomial coefficients
$A_i^j$ 
satisfy
$
\deg A_i^j+\lambda_j
=
\deg F_i
$.
Then, by employing the other major ingredients {\bf (2)},
{\bf (3)}, {\bf (4)} mentioned before, the Debarre Ampleness Conjecture finally turned into Theorem~\ref{Debarre Ampleness Conjecture}.

Recently, Brotbek and Darondeau~\cite{Brotbek-Lionel-2015} discovered a new way to construct 
negatively twisted symmetric differential forms
for a certain subfamily of hypersurfaces,
using pullbacks of some Pl\"ucker-embedding like morphisms,
and
they successfully controlled the base loci by means of
deep theorems in algebraic geometry. Their approach
together with the product coup gives another proof of the Debarre Ampleness Conjecture. Also, it is expected to achieve an effective lower bound on hypersurface degrees, which would ameliorate the preceding bound $N^{N^2}$ of Theorem~\ref{Debarre Ampleness Conjecture}.

\section{\bf Formal Matrices Produce Symmetric Differential Forms}
\label{Section: General Symmetric Differential Forms}
Aiming at the General Debarre Ampleness Conjecture, and
following the general strategy above,
we would like to 
first construct negatively twisted
symmetric differential forms.
Recalling the determinantal structure of Brotbek's symmetric differential forms~\cite[Lemma~4.5]{Brotbek-2014-arxiv}, in fact, we can take any
{\sl formal matrices} for construction,
regardless of negative twist at the moment.

Take an arbitrary scheme $\mathbf{P}$.
For any positive integers $1\leqslant n\leqslant e$,
for any $e$ line bundles $\mathcal{S}_1, \dots, 
\mathcal{S}_e$ over $\mathbf{P}$, 
we construct an $(e+n)\times (e+n)$
formal matrix $K$ such that, for $p=1\cdots e$ its $p$-th row consists of global sections $F_p^1, \dots, F_p^{e+n}\in \mathsf{H}^0(\mathbf{P}, \mathcal{S}_p)$,
and for $q=1\cdots n$ its $(e+q)$-th row is the {\sl formal
differential} --- to be defined --- of the $q$-th row:
\begin{equation}
\label{formal matrix K}
K
:=
\begin{pmatrix}
F_1^1 & \cdots & F_1^{e+n} \\
\vdots & & \vdots \\
F_e^1 & \cdots & F_e^{e+n} \\[6pt]
\dformal F_1^1 & \cdots & 
\dformal F_1^{e+n} \\
\vdots & & \vdots \\
\dformal F_n^1 & \cdots & 
\dformal F_n^{e+n} 
\end{pmatrix}.
\end{equation}

We will see later that the determinant of $K$ produces a twisted symmetric differential form on $\mathbf{P}$.
First of all, we define the above formal differential entries $\dformal F_i^j$ in a natural way.

\begin{Definition}
\label{define formal differential}
Let $\mathcal{S}$ be a line bundle over $\mathbf{P}$,
with a global section $S$. For any Zariski open set $U\subset\mathbf{P}$
with a trivialization
$
\mathcal{S}\big{\vert}_U
=
\mathcal{O}_U
\cdot
s$ ($s\in \mathsf{H}^0(U,\,\mathcal{S})$ is invertible),
 denote $S/s$ for the unique
$\widetilde{s}\in\mathcal{O}_{\mathbf{P}}(U)$
such that
$
S
=
\widetilde{s}\cdot s
$.
Also, define the formal differential $\dformal\,S$ in the local coordinate $(U,s)$ by:
\[
\dformal
S\,(U,s)\,
:=\,
\mathrm{d}\,(S/s)\cdot s\
\in\
\mathsf{H}^0(U,\,\Omega_{\mathbf{P}}^1\otimes \mathcal{S}),
\]
where `$\mathrm{d}$' stands for the usual differential.
\end{Definition}

Let us check that the above definition works well with the usual Leibniz's rule. Indeed,
let $\mathcal{S}_1$, $\mathcal{S}_2$ be two line bundles over $\mathbf{P}$,
with any two global sections $S_1$, $S_2$ respectively. For any Zariski open set $U\subset\mathbf{P}$
with trivializations
$
\mathcal{S}_1\big{\vert}_U
=
\mathcal{O}_U
\cdot
s_1$ and
$
\mathcal{S}_2\big{\vert}_U
=
\mathcal{O}_U
\cdot
s_2$,
we may compute:
\[
\aligned
\dformal
(S_1\otimes S_2)\,(U,s_1\otimes s_2)\,
&
=\,
\mathrm{d}\,
(S_1/s_1 \cdot S_2/s_2)\cdot s_1\otimes s_2
\\
&
=\,
\mathrm{d}\,
(S_1/s_1)\cdot 
S_2/s_2\,
\cdot\,
s_1\otimes s_2\,
+\,
\mathrm{d}\,
(S_2/s_2)\cdot 
S_1/s_1\,
\cdot\,
s_1\otimes s_2
\\
\explain{\,identify $S_1 \otimes S_2 \cong S_2 \otimes S_1$\,}
\qquad
&
=\,
\mathrm{d}\,
(S_1/s_1)\cdot 
s_1\,
\otimes\,S_2/s_2
\cdot s_2\,
+\,
\mathrm{d}\,
(S_2/s_2)\cdot 
s_2\,
\otimes\,
S_1/s_1
\cdot
s_1
\\
&
=\,
\dformal
S_1\,(U,s_1)\,
\otimes\,S_2\,
+\,
\dformal
S_2\,(U,s_2)\,
\otimes\,
S_1.
\endaligned
\]
Dropping the tensor symbol `$\otimes$'
and coordinates $(U,s_1,s_2)$, we abbreviate the above identity as:
\[
\dformal
(S_1\cdot S_2)\,
=\,
\dformal
S_1\cdot S_2
+
\dformal
S_2
\cdot 
S_1.
\]

Now, let us compute the determinant of~\thetag{\ref{formal matrix K}} in local coordinates.
For any Zariski open set $U\subset\mathbf{P}$
with trivializations
$
\mathcal{S}_1\big{\vert}_U
=
\mathcal{O}_U
\cdot
s_1$,
$\dots$,
$
\mathcal{S}_e\big{\vert}_U
=
\mathcal{O}_U
\cdot
s_e$,
writing $f_i^j:=F_i^j/s_i\in \mathcal{O}_{\mathbf{P}}(U)$,
we may factor:
{\footnotesize
\begin{equation}
\label{definition of formal matrix K}
K\,(U,s_1,\dots,s_e)\,
:=\,
\begin{pmatrix}
f_1^1\cdot s_1 & \cdots & f_1^{e+n}\cdot s_1 \\
\vdots & & \vdots \\
f_e^1\cdot s_e & \cdots & f_e^{e+n}\cdot s_e \\[6pt]
\mathrm{d}\,f_1^1\cdot s_1 & \cdots & 
\mathrm{d}\,f_1^{e+n}\cdot s_1 \\
\vdots & & \vdots \\
\mathrm{d}\,f_n^1\cdot s_n & \cdots & 
\mathrm{d}\,f_n^{e+n}\cdot s_n 
\end{pmatrix}\,
=\,
\underbrace{
\begin{pmatrix}
s_1 &&&&&
\\
& 
\ddots 
&&&&
\\
&&
s_e
&&&  
\\
&&&
s_1
&&
\\
&&&&
\ddots
&
\\
&&&&&
s_n
\end{pmatrix}
}_{
=:\,
\mathsf{T}_{s_1,\dots,s_e}^U
}\,
\cdot\,
\underbrace{
\begin{pmatrix}
f_1^1 & \cdots & f_1^{e+n} \\
\vdots & & \vdots \\
f_e^1 & \cdots & f_e^{e+n}\\[6pt]
\mathrm{d}\,f_1^1 & \cdots & 
\mathrm{d}\,f_1^{e+n} \\
\vdots & & \vdots \\
\mathrm{d}\,f_n^1 & \cdots & 
\mathrm{d}\,f_n^{e+n}
\end{pmatrix}
}_{
=:\,
(K)_{s_1,\dots,s_e}^U
}.
\end{equation}
}
Denoting the last two matrices by
$\mathsf{T}_{s_1,\dots,s_e}^U$, $(K)_{s_1,\dots,s_e}^U$, we obtain:
\begin{equation}
\label{definition of det K}
\aligned
\det
K\,(U,s_1,\dots,s_e)\,
&
=\,
\det
\mathsf{T}_{s_1,\dots,s_e}^U\,
\cdot\,
\det\,
(K)_{s_1,\dots,s_e}^U
\\
&
=\,
s_1
\cdots
s_e\,
s_1
\cdots
s_n\,
\cdot\,
\det 
\big(
K
\big)_{s_1,\dots,s_e}^U\ 
\in\
\mathsf{H}^0\,
\Big(
U,\,\mathsf{Sym}^n\,\Omega^1_{\mathbf{P}}\otimes 
\mathcal{S}(\heartsuit)\Big),
\endaligned
\end{equation}
where for shortness we denote:
\begin{equation}
\label{twisted degree of heart}
\mathcal{S}(\heartsuit)\,
:=\,
\big(
\otimes_{p=1}^{e}\,\mathcal{S}_p
\big)\,
\otimes\,
\big(
\otimes_{q=1}^{n}\,\mathcal{S}_q
\big).
\end{equation}

\begin{Proposition}
\label{very naive determinant of c+r very ample line bundle sections is still well defined}
The local definition: 
\[
\det
K\big{\vert}_U\,
:=\,
\det
K\,(U,s_1,\dots,s_e)\
\in\
\mathsf{H}^0\,\Big(U,\,\mathsf{Sym}^n\,\Omega^1_{\mathbf{P}}\otimes 
\mathcal{S}(\heartsuit)\Big)
\]
does not depend on the particular choices of invertible sections
$s_1,\dots,s_e$ over $U$.
\end{Proposition}

\begin{proof}
Assume that $s_1',\dots, s_e'$ are any other invertible sections
 of $\mathcal{S}_1\big{\vert}_{U}, \dots, \mathcal{S}_e\big{\vert}_{U}$.
Abbreviating the $p$-th row of the formal matrix $K$ by $F_p$,
we may compute:
\[
\aligned
F_p/s_p'\,
&
=\,
s_p/s_p'\,
\cdot\,
F_p/s_p
\qquad
\qquad
\qquad
\qquad
\qquad
\quad\ 
{\scriptstyle{(p\,=\,1\cdots\,e)}},
\\
\explain{Leibniz's rule}
\qquad
\mathrm{d}\,\big(F_q/s_q'\big)\, 
&
=\,
\mathrm{d}\,(s_q/s_q')\,
\cdot\,
F_q/s_q\,
+\,
s_q/s_q'\,
\cdot\,
\mathrm{d}\,\big(F_q/s_q\big)
\qquad
{\scriptstyle{(q\,=\,1\cdots\,n)}}.
\endaligned
\]
Thus we receive the transition identity:
\begin{equation}
\label{concrete transition relation, well-defined}
\big(
K
\big)_{s_1',\dots,s_e'}^U\,
=\,
\mathsf{T}_{s_1',\dots,s_e'}^{s_1,\dots,s_e}\,
\cdot\,
(K)_{s_1,\dots,s_e}^U,
\end{equation}
where $\mathsf{T}_{s_1',\dots,s_e'}^{s_1,\dots,s_e}$
is an $(e+n)\times (e+n)$ lower triangular matrix with the diagonal entries 
$s_1/s_1', \dots, s_e/s_e'$, $s_1/s_1', \dots, s_n/s_n'$
in the exact order. Taking determinant on both sides  of~\thetag{\ref{concrete transition relation, well-defined}} thus yields:
\[
\aligned
\det 
\big(
K
\big)_{s_1',\dots,s_e'}^U\,
&
=\,
\det\,
\mathsf{T}_{s_1',\dots,s_e'}^{s_1,\dots,s_e}\,
\cdot\,
\det 
\big(
K
\big)_{s_1,\dots,s_e}^U
\\
&
=\,
\underbrace{
\big(
\det\,
\mathsf{T}_{s_1',\dots,s_e'}^U
\big)^{-1}
}_{
=\,(
s_1'
\cdots
s_e'\,
s_1'
\cdots
s_n'
)^{-1}
}\,
\cdot\,
\underbrace{
\big(
\det\,
\mathsf{T}_{s_1,\dots,s_e}^U
\big)}_{
=\,
s_1
\cdots
s_e\,
s_1
\cdots
s_n
}\,
\cdot\,
\det 
\big(
K
\big)_{s_1,\dots,s_e}^U
\endaligned
\]
Multiplying by $\det\mathsf{T}_{s_1',\dots,s_e'}^U$ on both sides, we conclude the proof.
\end{proof}

Consequently, we receive

\begin{Proposition}
\label{det K is well defined}
The determinant of the formal matrix~\thetag{\ref{formal matrix K}} is globally well defined:
\[
\det
K\
\in\
\mathsf{H}^0\,\Big(\mathbf{P},\,\mathsf{Sym}^n\,\Omega^1_{\mathbf{P}}\otimes 
\mathcal{S}(\heartsuit)\Big).
\eqno
\qed
\]
\end{Proposition}

To grasp the essence of the above arguments, we provide
another wholly formal

\medskip
\noindent
{\em `Smart Proof'.}\,
Suppose that we do not know the meaning of {\sl formal differential} $\dformal F$,
for
any global section
$F$ of a line bundle $\mathcal{S}$ over $\mathbf{P}$.
Nevertheless, we still try to compute the determinant of the formal matrix~\thetag{\ref{formal matrix K}}.

First of all, we would like to extract some useful information out of the `mysterious' $\dformal F$.
{\em A priori},
we may assume that the formal differential satisfies
the Leibniz's rule in a certain sense, and also that when $\mathcal{S}=\mathcal{O}_{\mathbf{P}}$
it coincides with the usual differential $\mathrm{d}$.
Thus,
starting with any local section $z$ of $\mathcal{S}$,
we would have:
\[
\aligned
F\,
&
=\,
z
\cdot
F/z,
\\
\dformal
F\,
&
=\,
\dformal
z\,
\cdot
F/z
+
z
\cdot
\mathrm{d}\,
(F/z),
\endaligned
\]
that is:
\begin{equation}
\label{general matrix-entries-relation}
\begin{pmatrix}
F
\\
\dformal
F
\end{pmatrix}
=
\underline{
\begin{pmatrix}
z
&
0
\\
\star
&
z
\end{pmatrix}
}
\,
\begin{pmatrix}
F/z
\\
\mathrm{d}\,
\big(
F/z
\big)
\end{pmatrix},
\end{equation}
where 
$\star=\dformal z$
is meaningless\big/negligible in our coming computations.
Indeed, all we need is that the above underlined $2\times 2$ formal
matrix is lower triangular, with meaningful diagonal.

Back to our formal proof, we abbreviate every row of $K$ 
as $F_1,\dots,F_e, \dformal F_1, \dots, \dformal F_n$, and for convenience 
we write:
\[
K\,
=\,
\big(
F_1,
\dots,
F_e,
\dformal F_1, 
\dots,
\dformal F_n
\big)^{\mathrm{T}}.
\]
Over any Zariski open set $U\subset\mathbf{P}$
with invertible sections
$z_1,\dots, z_e$ of $\mathcal{S}_1, \dots, \mathcal{S}_e$ respectively,
using identity~\thetag{\ref{general matrix-entries-relation}},
we can dehomogenize $K$
with respect to $z_1,\dots,z_e$ by:
\begin{equation}
\label{dehomogenize K with respect to z_1, ..., z_e}
K\,
=\,
\mathsf{T}_{z_1,\dots,z_e}\,
\cdot\,
\underbrace{
\Big(
F_1/z_1,
\dots,
F_e/z_e,
\mathrm{d}\,\big(F_1/z_1\big), 
\dots,
\mathrm{d}\,\big(F_n/z_n\big)
\Big)^{\mathrm{T}}
}_{
=:\,(K)_{z_1,\dots,z_e}
},
\end{equation}
where $\mathsf{T}_{z_1,\dots,z_e}$
is a lower triangular $(e+n)\times (e+n)$ formal matrix
with diagonal entries $z_1,\dots, z_e, z_1, \dots, z_n$ in the exact order.
Now, it is desirable to notice that, on the right-hand-side of~\thetag{\ref{dehomogenize K with respect to z_1, ..., z_e}}, the matrix $(K)_{z_1,\dots,z_e}$ and
the diagonal of
the formal matrix $\mathsf{T}_{z_1,\dots,z_e}$ are well-defined,
thus all
the `mysterious differentials' of the matrix $K$
appear only in the strict lower-left part of  $\mathsf{T}_{z_1,\dots,z_e}$, which would immediately disappear after taking determinant
 on both sides of~\thetag{\ref{dehomogenize K with respect to z_1, ..., z_e}}:
\[
\aligned
\det\,
K\big{\vert}_U\,
&
=\,
\det\,
\mathsf{T}_{z_1,\dots,z_e}\,
\cdot\,
\det 
\big(
K
\big)_{z_1,\dots,z_e}
\\
&
=\,
z_1
\cdots
z_e\,
z_1
\cdots
z_n\,
\cdot\,
\det 
\big(
K
\big)_{z_1,\dots,z_e}\ 
\in\
\mathsf{H}^0\,
\Big(
U,\,\mathsf{Sym}^n\,\Omega^1_{\mathbf{P}}\otimes 
\mathcal{S}(\heartsuit)\Big).
\endaligned
\]

{\em Bien s\^ur}, it is independent of the choices of $z_1,\dots,z_e$, since the left-hand-side --- a formal determinant --- is.
\qed

\begin{Remark}
The formal differential $\dformal$
is much the same as the usual differential $\mathrm{d}$, in the sense that both of them can be defined locally, and both of them obey the Leibniz's rule.
These two facts constitute the essence of Proposition~\ref{det K is well defined}.
\end{Remark}

Next, we consider $e$ sections:
\begin{equation}
\label{the most general Fermat-type hypersurface}
F_i
=
\sum_{j=0}^{e+n}\,
F_i^j\
\in\
\mathsf{H}^0\,
(
\mathbf{P},
\mathcal{S}_i
)
\qquad
{\scriptstyle{(i\,=\,1\cdots\,e)}},
\end{equation}
each $F_i$ being the sum of $e+n+1$ global sections of the same line bundle $\mathcal{S}_i$.
Let $V$ be the intersection of the zero loci of the first $n$ sections:
\[
V\,
:=\,
\{
F_1
=
0
\}
\cap
\cdots
\cap
\{
F_n
=
0
\}\
\subset\
\mathbf{P},
\]
and let $X$ be the intersection of the zero loci of all the $e\geqslant n$ sections:
\[
X\,
:=\,
\{
F_1
=
0
\}
\cap
\cdots
\cap
\{
F_e
=
0
\}\
\subset\
V\
\subset
\mathbf{P}.
\]
Let $\mathsf{K}$ be the $(e+n)\times (e+n+1)$ formal matrix
whose $e+n$ rows copy the
$e+n+1$ terms of 
$F_1,\dots,F_e, \dformal F_1,
\dots, \dformal F_n$ in the exact order:
\[
\mathsf{K}
:=
\begin{pmatrix}
F_1^0 & \cdots & F_1^{e+n} \\
\vdots & & \vdots \\
F_e^0 & \cdots & F_e^{e+n} \\[6pt]
\dformal F_1^0 & \cdots & 
\dformal F_1^{e+n} \\
\vdots & & \vdots \\
\dformal F_n^0 & \cdots & 
\dformal F_n^{e+n} 
\end{pmatrix}.
\]
Also, for $j=0\cdots e+n$, let
$\widehat{\mathsf{K}}_{j}$ denote the submatrix of
$\mathsf{K}$ obtained by omitting the $(j+1)$-th column.

Since the restricted cotangent sheaf $\Omega_V^1\big{\vert}_X$
is formally defined by the $e+n$ equations: 
\[
F_1=0, \dots, F_{e}=0,\,
\dformal F_1=0, \dots, \dformal F_n=0,
\]
i.e. the sum of all $e+n+1$ columns of $\mathsf{K}$ vanishes, by Observation~\ref{naive Cramer's rule} below, we may receive

\begin{Proposition}
\label{naive forms coincide 3.3}
For all $j=0\cdots e+n$, 
the $e+n+1$ sections: 
\[
\psi_{j}\,
=\,
(-1)^j
\det\,
\widehat {\sf K}_{j}\
\in\
\mathsf{H}^0\,\Big(\mathbf{P},\,\mathsf{Sym}^n\,\Omega^1_{\mathbf{P}}\otimes 
\mathcal{S}(\heartsuit)\Big),
\]
when restricted to $X$, give one and the same section:
\[
\psi\
\in\
\mathsf{H}^0\,
\Big(
X,\,
\mathsf{Sym}^n\,\Omega^1_{V}\otimes 
\mathcal{S}(\heartsuit)\Big).
\] 
\end{Proposition}

\begin{Observation}
\label{naive Cramer's rule}
In a commutative ring $R$, for all positive integers $N\geqslant 1$, let
$
A^0,
A^1,
\dots,
A^N
\in 
R^N
$ 
be $N+1$ column vectors satisfying:
\[
A^0\,
+
A^1\,
+
\cdots
+
A^N\,
=\,
\mathbf{0}.
\]
Then for all $0\leqslant j_1, j_2 \leqslant N$, there hold the identities:
\[
(-1)^{j_1}
\det\,
\big(
A^0,\dots,\widehat{A^{j_1}},\dots,A^N
\big)\,
=\,
(-1)^{j_2}
\det\,
\big(A^0,\dots,\widehat{A^{j_2}},\dots,A^N
\big).
\eqno
\qed
\]
\end{Observation}

\begin{proof}[Proof of Proposition~\ref{naive forms coincide 3.3}]
Using the same notation as in~\thetag{\ref{definition of det K}}, for $j=0\cdots e+n$,
we obtain an $(e+n)\times(e+n)$ matrix
$(\widehat{\sf K}_j)_{s_1,\dots,s_e}^U$. We also define:
\[
\mathsf{K}\,(U,s_1,\dots,s_e)\,
:=\,
\mathsf{T}_{s_1,\dots,s_e}^U\,
\cdot\,
({\sf K})_{s_1,\dots,s_e}^U,
\]
where the $(e+n)\times(e+n+1)$ matrix $({\sf K})_{z_1,\dots,z_e}^U$
satisfies that, for $j=0\cdots e+n$,
the matrix 
$(\widehat{\sf K}_j)_{s_1,\dots,s_e}^U$
is obtained by omitting the
$(j+1)$-th column of $({\sf K})_{s_1,\dots,s_e}^U$.

We may view all entries of
$({\sf K})_{s_1,\dots,s_e}^U$
as sections in $\mathsf{H}^0(U\cap X, \mathsf{Sym}^{\bullet}\,\Omega_V^1)$, where:
\[
\mathsf{Sym}^{\bullet}\,\Omega_V^1\,
:=\,
\oplus_{k=0}^{\infty}\,
\mathsf{Sym}^{k}\,\Omega_V^1.
\]
Thus the sum of all columns of $({\sf K})_{s_1,\dots,s_e}^U$ vanishes, and hence Observation~\ref{naive Cramer's rule} yields:
\[
(-1)^{j_1}
\det\,
(\widehat{\sf K}_{j_1})_{s_1,\dots,s_e}^U\,
=\,
(-1)^{j_2}
\det\,
(\widehat{\sf K}_{j_2})_{s_1,\dots,s_e}^U\
\in\
\mathsf{H}^0(U\cap X, \mathsf{Sym}^{n}\,\Omega_V^1)
\qquad
{\scriptstyle{(j_1,\,j_2\,=\,0\,\cdots\,e+n)}}.
\]
By multiplication of $\det
\mathsf{T}_{s_1,\dots,s_e}^U$ on both sides, we conclude the proof.
\end{proof}

Remember that our goal is to construct negatively twisted
symmetric differential forms. 
One idea,
foreshadowed by the constructions in~\cite{Brotbek-2014-arxiv, Xie-2015-arxiv}, 
is to find some $e+n+1$ line bundles
$\mathcal{T}_0$, $\dots$, $\mathcal{T}_{e+n}$ with respective global sections
$t_0, \dots, t_{e+n}$ 
having empty base locus, 
such that the line bundle:
\[
\mathcal{S}(\heartsuit)
\otimes
\mathcal{T}_0^{-1}\otimes
\cdots
\otimes
\mathcal{T}_{e+n}^{-1}\,
=:\,
\mathcal{S}(\heartsuit')
\
<\
0,
\]
is negative,
and
such that:
\begin{equation}
\label{general idea}
\widehat{\omega}_{j}\,
:=\,
\frac{\psi_{j}}
{t_0\cdots t_{e+n}}\,
=\,
\frac{(-1)^j}
{t_0\cdots t_{e+n}}\,
\det\,
\widehat {\sf K}_{j}\
\in\
\mathsf{H}^0\,
\Big(
\mathsf{D}(t_j),\,
\mathsf{Sym}^n\,\Omega^1_{\mathbf{P}}\otimes 
\mathcal{S}(\heartsuit')
\Big)
\qquad
{\scriptstyle{(j\,=\,0\,\cdots\,e+n)}}
\end{equation}
have no poles. Then, these $e+n+1$ sections, restricted to $X$, would
glue together to make a global negatively twisted symmetric differential form:
\[
\omega\
\in\
\mathsf{H}^0\,
\Big(
X,\,
\mathsf{Sym}^n\,\Omega^1_{V}\otimes 
\mathcal{S}(\heartsuit')
\Big).
\]

For the purpose of~\thetag{\ref{general idea}},
we may require that
every $t_0,\dots, t_{e+n}$ subsequently `divides' the corresponding column
of ${\sf K}$ in the exact order. With some additional effort, we shall make this idea rigorous 
in our central applications.

\section{\bf A Dividing Trick}
\label{section: A Dividing Trick}
Let
$\mathcal{L}$ be a line bundle over $\mathbf{P}$ such that it has
$N+1$ global sections $\zeta_0,\dots,\zeta_N$ having empty common base locus.
Let $c\geqslant 1$, $r\geqslant 0$ be two integers with $2c+r\geqslant N$ and $c+r< N$. 
Let $\mathcal{A}_1,\dots,\mathcal{A}_{c+r}$ be
$c+r$ auxiliary line bundles to be determined.
Now, 
we consider
$c+r$ {\sl Fermat-type} sections having the same shape as~\thetag{\ref{General c+r Fermat type hypersurfaces}}:
\begin{equation}
\label{general Fermat-type sections for c+r line bundles}
F_i
=
\sum_{j=0}^N\,
\underline{
A_i^j
}
\,\zeta_j^{\lambda_j}\
\in\
\underline{
\mathcal{A}_i
\otimes
\mathcal{L}^{\epsilon_i^j}
}
\otimes
\mathcal{L}^{\lambda_j}
\,
=\,
\underbrace{
\mathcal{A}_i
\otimes
\mathcal{L}^{d_i}
}_{
=\,
\mathcal{S}_i
}
\qquad
{\scriptstyle{(i\,=\,1\cdots\,c+r)}},
\end{equation}
 where $\epsilon_i^j, \lambda_j, d_i\geqslant 1$ are integers 
satisfying 
$
\epsilon_i^j
+
\lambda_j
=d_i$,
and where every $A_i^j$ is some global section of
the line bundle 
$
\mathcal{A}_i
\otimes
\mathcal{L}^{\epsilon_i^j}$.

For the first $c$ equations of~\thetag{\ref{general Fermat-type sections for c+r line bundles}}, 
a formal differentiation yields:
\begin{equation}
\label{differentiating general Fermat-type sections}
\dformal
F_i\,
=\,
\sum_{j=0}^N\,
\dformal
\big(
A_i^j
\,\zeta_j^{\lambda_j}
\big)\,
=\,
\sum_{j=0}^N\,
\zeta_j^{\lambda_j-1}\,
\underbrace{
\big(
\zeta_j\,
\dformal
A_i^j
+
\lambda_j\,
A_i^j\,
\dformal
\zeta_j
\big)
}_{
=:\,\mathsf{B}_i^j
}
\qquad
{\scriptstyle{(i\,=\,1\cdots\,c)}}.
\end{equation}

Now, we construct the $(c+r+c) \times (N+1)$ matrix ${\sf M}$,
whose first $c+r$ rows
consist of all $(N+1)$ terms in the expressions~\thetag{\ref{general Fermat-type sections for c+r line bundles}} of $F_1,\dots,F_{c+r}$ in the exact order, 
and whose last $c$ rows consist of all $(N+1)$ terms in the expressions~\thetag{\ref{differentiating general Fermat-type sections}} of
$\dformal F_1,\dots,\dformal F_c$ in the exact order:
\begin{equation}
\label{matrix K = ...}
{\sf M}\,
:=\,
\begin{pmatrix}
A_1^0\,\zeta_0^{\lambda_0} & \cdots & A_1^N\,\zeta_N^{\lambda_N} \\
\vdots & & \vdots \\
A_{c+r}^0\,\zeta_0^{\lambda_0} & \cdots & A_{c+r}^N\,\zeta_N^{\lambda_N} \\[6pt]
\dformal \big(A_1^0\,\zeta_0^{\lambda_0}\big)
 & \cdots & 
\dformal \big(A_1^N\,\zeta_N^{\lambda_N}\big) \\
\vdots & & \vdots \\
\dformal \big(A_c^0\,\zeta_0^{\lambda_0}\big)
 & \cdots & 
\dformal \big(A_c^N\,\zeta_N^{\lambda_N}\big) \\
\end{pmatrix}\,
=\,
\begin{pmatrix}
A_1^0\,\zeta_0^{\lambda_0} & \cdots & A_1^N\,\zeta_N^{\lambda_N} \\
\vdots & & \vdots \\
A_{c+r}^0\,\zeta_0^{\lambda_0} & \cdots & A_{c+r}^N\,\zeta_N^{\lambda_N} \\[6pt]
{\sf B}_1^0\,\zeta_0^{\lambda_0-1} & \cdots & {\sf B}_1^N\,\zeta_N^{\lambda_N-1} \\
\vdots & & \vdots \\
{\sf B}_c^0\,\zeta_0^{\lambda_0-1} & \cdots & {\sf B}_c^N\,\zeta_N^{\lambda_N-1} \\
\end{pmatrix}.
\end{equation}

Denote $n:=N-c-r$, observe that $1\leqslant n\leqslant c$.
For every $1 \leqslant j_1 <
\dots < j_n \leqslant c$, denote by ${\sf M}_{j_1,\dots,j_n}$ the
$(c+r+n)\times (N+1)$ submatrix of ${\sf M}$ consisting of the first upper $c+r$
rows and the selected rows $c+r+j_1,\dots,c+r+j_n$. Also, for $j=0\cdots N$, denote by
$\widehat{{\sf M}}_{j_1,\dots,j_n;\,j}$ the submatrix of
${\sf M}_{j_1,\dots,j_n}$ obtained by omitting the $(j+1)$-th column.

Let $V\subset \mathbf{P}$ be the subvariety defined by
the first $c$ sections $F_1,\dots,F_c$, and let
$X\subset \mathbf{P}$ be the subvariety defined by
all the $c+r$ sections $F_1,\dots,F_{c+r}$.
Now, applying Proposition~\ref{naive forms coincide 3.3}, 
denoting:
\[
\mathcal{A}_{j_1,\dots,j_n}^{1,\dots,c+r}\,
:=\,
\mathcal{A}_{1}
\otimes
\cdots
\otimes
\mathcal{A}_{c+r}
\otimes
\mathcal{A}_{j_1}
\otimes
\cdots
\otimes
\mathcal{A}_{j_n},
\]
we receive

\begin{Proposition}
\label{corollary of naive forms}
For every $1 \leqslant j_1 <
\dots < j_n \leqslant c$,
for all $j=0\cdots N$, 
the $N+1$ sections: 
\[
\psi_{j_1,\dots,j_n;\,j}\,
=\,
(-1)^j
\det\,
\widehat{{\sf M}}_{j_1,\dots,j_n;\,j}\
\in\
\mathsf{H}^0\,\Big(\mathbf{P},\,\mathsf{Sym}^n\,\Omega^1_{\mathbf{P}}\otimes 
\mathcal{A}_{j_1,\dots,j_n}^{1,\dots,c+r}
\otimes 
\mathcal{L}^{\heartsuit_{j_1,\dots,j_n}}\Big),
\]
when restricted to $X$, give one and the same symmetric differential form:
\[
\psi_{j_1,\dots,j_n}\
\in\
\mathsf{H}^0\,
\Big(
X,\,
\mathsf{Sym}^n\,\Omega^1_{V}
\otimes
\mathcal{A}_{j_1,\dots,j_n}^{1,\dots,c+r}
\otimes 
\mathcal{L}^{\heartsuit_{j_1,\dots,j_n}}\Big),
\]
with the twisted degree:
\begin{equation}
\label{twisted degree of heart}
\heartsuit_{j_1,\dots,j_n}\,
=\,
\sum_{p=1}^{c+r}\,d_{p}
+
\sum_{q=1}^{n}\,d_{j_q}.
\end{equation} 
\end{Proposition}

Observe in~\thetag{\ref{matrix K = ...}} that the $N+1$ columns of $\mathsf{M}$ are subsequently divisible by
$\zeta_0^{\lambda_0-1}, \dots, \zeta_N^{\lambda_N-1}$.
Dividing out these factors, we receive the formal matrix:
\[
{\sf C}
:=
\begin{pmatrix}
A_1^0\,\zeta_0 & \cdots & A_1^N\,\zeta_N \\
\vdots & & \vdots \\
A_{c+r}^0\,\zeta_0 & \cdots & A_{c+r}^N\,\zeta_N \\[6pt]
{\sf B}_1^0 & \cdots & {\sf B}_1^N \\
\vdots & & \vdots \\
{\sf B}_c^0 & \cdots & {\sf B}_c^N \\
\end{pmatrix}.
\]
By mimicking the notation of the submatrices 
${\sf M}_{j_1,\dots,j_n}$ ,
$\widehat{{\sf M}}_{j_1,\dots,j_n;\,j}$
of ${\sf M}$, we analogously define
the submatrices
${\sf C}_{j_1,\dots,j_n}$ ,
$\widehat{{\sf C}}_{j_1,\dots,j_n;\,j}$ of ${\sf C}$.

Now, we interpret Proposition~\ref{corollary of naive forms} in terms of the matrix $\mathsf{C}$, starting by the formal computation:
\[
\aligned
(-1)^{j}
\det\,
\widehat{{\sf M}}_{j_1,\dots,j_n;\,j}\,
&
=\,
(-1)^{j}\,
\zeta_0^{\lambda_0-1}
\cdots\,
\widehat{
\zeta_{j}^{\lambda_{j}-1}
} 
\cdots\,
\zeta_N^{\lambda_N-1}\,
\det\,
\widehat{{\sf C}}_{j_1,\dots,j_n;\,j}
\\
&
=\,
\frac
{
(-1)^{j}
}
{
\zeta_{j}^{\lambda_{j}-1}
}\,
\det\,
\widehat{{\sf C}}_{j_1,\dots,j_n;\,j}\,
\cdot\,
\underline{
\zeta_0^{\lambda_0-1}
\cdots\,
\zeta_N^{\lambda_N-1}
}.
\endaligned
\]
Dividing by
$\underline{\zeta_0^{\lambda_0-1}
\cdots\,
\zeta_N^{\lambda_N-1}}$ on both sides above, we receive the following $N+1$ `coinciding' forms:
\begin{equation}
\label{definition of Brotbek's symmetric differential forms}
\underbrace{
\frac
{
(-1)^{j}
\det\,
\widehat{{\sf M}}_{j_1,\dots,j_n;\,j}
}
{
\zeta_0^{\lambda_0-1}
\cdots\,
\zeta_N^{\lambda_N-1}
}
}_{
\text{independent of }j
}\,
=\,
\underbrace{
\frac
{
(-1)^{j}
}
{
\zeta_{j}^{\lambda_{j}-1}
}\,
\det\,
\widehat{{\sf C}}_{j_1,\dots,j_n;\,j}
}_{
\text{have no pole over }
\mathsf{D}(\zeta_j)
}
\qquad
{\scriptstyle{(j\,=\,0\,\cdots\,N)}}.
\end{equation}
This is the aforementioned {\sl dividing trick}.

\begin{Proposition}
\label{the symmetric differential forms with many very ample line bundle}
For all $j=0\cdots N$, the formal symmetric differential forms:
\[
\widehat\omega_{j_1,\dots,j_n;\,j}\,
=\,
\frac
{
(-1)^{j}
\det\,
\widehat{{\sf M}}_{j_1,\dots,j_n;\,j}
}
{
\zeta_0^{\lambda_0-1}
\cdots\,
\zeta_N^{\lambda_N-1}
}
\]
are well-defined sections in:
\[
\mathsf{H}^0\,
\Big(
\mathsf{D}(\zeta_j),\,
\mathsf{Sym}^n\,\Omega^1_{\mathbf{P}}
\otimes
\mathcal{A}_{j_1,\dots,j_n}^{1,\dots,c+r}
\otimes 
\mathcal{L}^{\heartsuit_{j_1,\dots,j_n}'}\Big),
\]
with the twisted degree:
\[
\heartsuit_{j_1,\dots,j_n}'\,
:=\,
\sum_{p=1}^{c+r}\,d_{p}
+
\sum_{q=1}^{n}\,d_{j_q}
-
\sum_{k=0}^{N}\,(\lambda_k-1).
\]
Moreover, when restricted to $X$, they
glue together to make a global section:
\[
\omega_{j_1,\dots,j_n}\
\in\
\mathsf{H}^0\,
\Big(
X,\,
\mathsf{Sym}^n\,\Omega^1_{V}
\otimes
\mathcal{A}_{j_1,\dots,j_n}^{1,\dots,c+r}
\otimes 
\mathcal{L}^{\heartsuit_{j_1,\dots,j_n}'}\Big).
\]
\end{Proposition}

While the formal identity~\thetag{\ref{definition of Brotbek's symmetric differential forms}} transparently shows the essence of this proposition, it is not yet a proof
by itself, since both sides are to be defined. Indeed, to bypass the potential trouble of divisibility, the rigorous proof below is much more involved than one would first expect.

\begin{proof}
Without loss of generality, we only prove the case
$j_1=1,\dots,j_n=n$,
and we will often drop the indices $j_1,\dots,j_n$,
since no confusion could occur.
Here is a sketch of the proof.

\medskip
\noindent
{\em Step 1.}
Over each Zariski open set
$U\subset \mathsf{D}(\zeta_j)$
with trivializations
$
\mathcal{A}_1\big{\vert}_U
=
\mathcal{O}_U
\cdot
a_1$,
$\dots$,
$
\mathcal{A}_{c+r}\big{\vert}_U
=
\mathcal{O}_U
\cdot
a_{a+r}$, 
we compute the expression of
$\widehat\omega_{j}:=\widehat\omega_{j_1,\dots,j_n;\,j}$
in coordinates $(U,a_1,\dots,a_{c+r},\zeta_j)$. 

\medskip
\noindent
{\em Step 2.}
We show that the obtained symmetric form $\widehat\omega_{j}\big{\vert}_U$
is independent of the choices of trivializations
$a_1,\dots,a_{c+r}$, whence we conclude the first claim.

\medskip
\noindent
{\em Step 3.}
For any distinct indices $0\leqslant \ell_1, \ell_2\leqslant N$,
over any Zariski open set
$U\subset \mathsf{D}(\zeta_{\ell_1})\cap\mathsf{D}(\zeta_{\ell_2})$
with trivializations
$
\mathcal{A}_1\big{\vert}_U
=
\mathcal{O}_U
\cdot
a_1$,
$\dots$,
$
\mathcal{A}_{c+r}\big{\vert}_U
=
\mathcal{O}_U
\cdot
a_{a+r}$, 
we show that:
\begin{equation}
\label{show that the target form glue together}
\widehat\omega_{\ell_1}\,
=\,
\widehat\omega_{\ell_2}\
\in\
\mathsf{H}^0\,
\Big(
U,\,
\mathsf{Sym}^n\,\Omega^1_{V}
\otimes
\mathcal{A}_{j_1,\dots,j_n}^{1,\dots,c+r}
\otimes 
\mathcal{L}^{\heartsuit_{j_1,\dots,j_n}'}\Big)
\end{equation}
by computations in coordinates.
Thus we conclude the second claim.

\medskip
\noindent
{\em Proof of Step 1.}
Recalling~\thetag{\ref{definition of formal matrix K}},
by trivializations: 
\[
\underbrace{
\mathcal{A}_i
\otimes
\mathcal{L}^{d_i}
}_{=:\,\mathcal{S}_i}\big{\vert}_U\,
=\,
\mathcal{O}_U
\cdot 
\underbrace{
a_i\,\zeta_j^{d_i}
}_{
=:\,s_i
}
\qquad
{\scriptstyle{(i\,=\,1\cdots\,c+r)}},
\]
the formal matrix $\mathsf{K}:=\mathsf{M}_{j_1,\dots,j_n}$ has coordinates:
\begin{equation}
\label{very formal reltion}
\mathsf{K}\,
=\,
\mathsf{T}_{a_1,\dots,a_{c+r}}^{\zeta_j}\,
\cdot\,
\underbrace{
\Big(
F_1/s_1,
\dots,
F_{c+r}/s_{c+r},
\mathrm{d}\,\big(F_1/s_1\big), 
\dots,
\mathrm{d}\,\big(F_n/s_n\big)
\Big)^{\mathrm{T}}
}_{
=:\,(\mathsf{K})_{a_1,\dots,a_{c+r}}^{\zeta_j}
},
\end{equation}
where $\mathsf{T}_{a_1,\dots,a_{c+r}}^{\zeta_j}$
is an $N\times N$ diagonal matrix
with the diagonal $s_1,\dots, s_{c+r}, s_1, \dots, s_n$,
and where like~\thetag{\ref{dehomogenize K with respect to z_1, ..., z_e}} we abbreviate the first $c+r$ rows of $\mathsf{K}$ by $F_1,\dots,F_{c+r}$.
Further computation yields:
\[
F_i\,/\,s_i\,
:=\,
\big(
A_i^0
\,\zeta_0^{\lambda_0},
\dots,
A_i^N
\,\zeta_N^{\lambda_N}
\big)\,/\,s_i
=\,
\Big(
{A_i^0}/
{\alpha_i^0}
\cdot
(\zeta_0/\zeta_{j})^{\lambda_0},
\dots,
{A_i^N}
/
{\alpha_i^N}
\cdot
(\zeta_N/\zeta_{j})^{\lambda_N}
\Big)
\qquad
{\scriptstyle{(i\,=\,0\,\cdots\,c+r)}},
\]
where $\alpha_i^j:=a_i\,\cdot\, \zeta_j^{\epsilon_i^j}$.
`Dividing' every column of 
$(\mathsf{K})_{a_1,\dots,a_{c+r}}^{\zeta_j}$ subsequently
by $(\zeta_0/\zeta_{j})^{\lambda_0-1},\dots,(\zeta_N/\zeta_{j})^{\lambda_N-1}$,
we obtain an $N\times (N+1)$ matrix
$
(\mathsf{C})_{a_1,\dots,a_{c+r}}^{\zeta_j}
$. 
For every $\ell=0\cdots N$, we denote by 
$(\widehat{\mathsf{C}}_{\ell})_{a_1,\dots,a_{c+r}}^{\zeta_j}$
the submatrix of 
$
(\mathsf{C})_{a_1,\dots,a_{c+r}}^{\zeta_j}
$
obtained by deleting its $(\ell+1)$-th column.
Now, formula~\thetag{\ref{definition of det K}} yields:
\begin{equation}
\label{easy formula below (17)}
\det\,
\widehat{{\sf M}}_{j_1,\dots,j_n;\,j}\,
=\,
s_1
\cdots
s_{c+r}\,
s_1
\cdots
s_n\,
(\zeta_0/\zeta_{j})^{\lambda_0-1}
\cdots
(\zeta_N/\zeta_{j})^{\lambda_N-1}\,
\det\,
(\widehat{\mathsf{C}}_{j})_{a_1,\dots,a_{c+r}}^{\zeta_j}.
\end{equation}
Thus, in coordinates $(U,a_1,\dots,a_{c+r},\zeta_j)$, we obtain\big/define:
{\footnotesize
\begin{equation}
\label{definition of the target symmetric form}
\aligned
\widehat\omega_{j}\,
&
=\,
\frac
{
(-1)^{j}
\det\,
\widehat{{\sf M}}_{j_1,\dots,j_n;\,j}
}
{
\zeta_0^{\lambda_0-1}
\cdots\,
\zeta_N^{\lambda_N-1}
}\,
:=\,
(-1)^{j}\,
\frac
{s_1
\cdots
s_{c+r}\,
s_1
\cdots
s_n}
{\zeta_{j}^{(\lambda_0-1)+\cdots
+(\lambda_N-1)}}\,
\cdot\,
\det\,
(\widehat{\mathsf{C}}_{j})_{a_1,\dots,a_{c+r}}^{\zeta_j}
\\
&
=\,
(-1)^{j}\,
a_1
\cdots
a_{c+r}
\cdot
a_1
\cdots
a_n\,
\cdot\,
\zeta_j^{\heartsuit_{j_1,\dots,j_n}'}\,
\cdot\,
\det\,
(\widehat{\mathsf{C}}_{j})_{a_1,\dots,a_{c+r}}^{\zeta_j}
\\
&
\in\
\mathsf{H}^0\,
\Big(
U,\,
\mathsf{Sym}^n\,\Omega^1_{\mathbf{P}}
\otimes
\mathcal{A}_{j_1,\dots,j_n}^{1,\dots,c+r}
\otimes 
\mathcal{L}^{\heartsuit_{j_1,\dots,j_n}'}\Big).
\endaligned
\end{equation}
}

\medskip
\noindent
{\em Proof of Step 2.}
We only need to show that:
\[
a_1
\cdots
a_{c+r}\,
a_1
\cdots
a_n\,
\cdot\,
\det\,
(\widehat{\mathsf{C}}_{j})_{a_1,\dots,a_{c+r}}^{\zeta_j}\
\in\
\mathsf{H}^0\,
\Big(
U,\,
\mathsf{Sym}^n\,\Omega^1_{\bf P}
\otimes
\mathcal{A}_{j_1,\dots,j_n}^{1,\dots,c+r}
\Big)
\]
is independent of the choices of $a_1,\dots,a_{c+r}$.

Let $\widetilde{a}_1,\dots,\widetilde{a}_{c+r}$ be any other choices of
invertible sections of $\mathcal{A}_1\big{\vert}_U, \dots,\mathcal{A}_{c+r}\big{\vert}_U$. Accordingly, we
 obtain the  matrices 
$({\mathsf{C}})_{\widetilde{a}_1,\dots,\widetilde{a}_{c+r}}^{\zeta_j}$, $(\widehat{{\mathsf{C}}}_{j})_{\widetilde{a}_1,\dots,\widetilde{a}_{c+r}}^{\zeta_j}$, and we denote $\widetilde{\alpha}_i^j:=\widetilde{a}_i\,\cdot\, \zeta_j^{\epsilon_i^j}$.
Then, for $i=1\cdots c+r$, the $i$-th row of the  matrix $({\mathsf{C}})_{\widetilde{a}_1,\dots,\widetilde{a}_{c+r}}^{\zeta_j}$ is:
{\footnotesize
\[
\explain{
$
{\alpha}_i^k/\widetilde{\alpha}_i^k
=
a_i/\widetilde{a}_i
$\,}
\qquad
\Big(
A_i^0/
\widetilde{\alpha}_i^0
\cdot
(\zeta_0/\zeta_{j}),
\dots,
A_i^N/
\widetilde{\alpha}_i^N
\cdot
(\zeta_N/\zeta_{j})
\Big)\,
=\,
a_i/\widetilde{a}_i\,
\cdot\,
\underbrace{
\Big(
A_i^0/
{\alpha}_i^0
\cdot
(\zeta_0/\zeta_{j}),
\dots,
A_i^N/
{\alpha}_i^N
\cdot
(\zeta_N/\zeta_{j})
\Big)
}_{
\text{the } i\text{-th row
of the matrix }
(\mathsf{C})_{a_1,\dots,a_{c+r}}^{\zeta_j}
}.
\]
}\!\!
Also, for $i=1\cdots n$, $k=0\cdots N$, 
using:
\[
\mathrm{d}\,
({A_i^k}/
{\widetilde{\alpha}}_i^k)\,
=\,
\mathrm{d}\,
({A_i^k}/
{\alpha_i^k}
\cdot
a_i/\widetilde{a}_i
)
=
a_i/\widetilde{a}_i
\cdot
\mathrm{d}\,
({A_i^k}/
{\alpha_i^k}
)
+
{A_i^k}/
{\alpha_i^k}
\cdot
\mathrm{d}\,
(
a_i/\widetilde{a}_i
),
\]
we see that the $(c+r+i, k+1)$-th entry
of $({\mathsf{C}})_{\widetilde{a}_1,\dots,\widetilde{a}_{c+r}}^{\zeta_j}$ satisfies:
\begin{equation}
\label{compute something}
\aligned
&
(\zeta_k/\zeta_{j})
\cdot
\mathrm{d}\,
({A_i^k}/
{\widetilde{\alpha}_i^k})
+
\lambda_k\,
({A_i^k}/
{\widetilde{\alpha}_i^k})
\cdot
\mathrm{d}\,
(\zeta_k/\zeta_{j})\,
\\
=\,\,
&
a_i/\widetilde{a}_i
\cdot
\underbrace{
\big[
(\zeta_k/\zeta_{j})
\cdot
\mathrm{d}\,
({A_i^k}/
{\alpha_i^k})
+
\lambda_k\,
({A_i^k}/
{\alpha_i^k})
\cdot
\mathrm{d}\,
(\zeta_k/\zeta_{j})
\big]
}_{
=\,
\text{the } (c+r+i,\,k+1)\text{-th entry
of } (\mathsf{C})_{a_1,\dots,a_{c+r}}^{\zeta_j}
}\,
+\,
\mathrm{d}\,
(a_i/\widetilde{a}_i)
\cdot
\underbrace{
{A_i^k}/
{\alpha_i^k}
\cdot
(\zeta_k/\zeta_{j})
}_{
=\,
(i,\,k+1)\text{-th entry}
}.
\endaligned
\end{equation}
Hence we receive the transition identity:
\[
({\mathsf{C}})_{\widetilde{a}_1,\dots,\widetilde{a}_{c+r}}^{\zeta_j}\,
=\,
\mathsf{T}_{\widetilde{a}_1,\dots,\widetilde{a}_{c+r}}^{a_1,\dots,a_{c+r}}\,
\cdot\,
(\mathsf{C})_{a_1,\dots,a_{c+r}}^{\zeta_j},
\]
where $\mathsf{T}_{\widetilde{a}_1,\dots,\widetilde{a}_{c+r}}^{a_1,\dots,a_{c+r}}$ is a lower triangular matrix
with the product of the diagonal:
\[
\det\,
\mathsf{T}_{\widetilde{a}_1,\dots,\widetilde{a}_{c+r}}^{a_1,\dots,a_{c+r}}\,
=\,
\frac
{
a_1
\cdots 
a_{c+r}
\cdot
a_1
\cdots 
a_{n}
}
{
\widetilde{a}_1
\cdots 
\widetilde{a}_{c+r}
\cdot
\widetilde{a}_1
\cdots 
\widetilde{a}_n
}.
\] 
In particular, we have:
\[
(\widehat{{\mathsf{C}}}_{j})_{\widetilde{a}_1,\dots,\widetilde{a}_{c+r}}^{\zeta_j}\,
=\,
\mathsf{T}_{\widetilde{a}_1,\dots,\widetilde{a}_{c+r}}^{a_1,\dots,a_{c+r}}\,
\cdot\,
(\widehat{\mathsf{C}}_{j})_{a_1,\dots,a_{c+r}}^{\zeta_j},
\]
hence, by taking determinant on both sides above, we obtain:
\[
\widetilde{a}_1
\cdots 
\widetilde{a}_{c+r}
\cdot
\widetilde{a}_1
\cdots 
\widetilde{a}_n\,
\cdot\,
\det\,
(\widehat{{\mathsf{C}}}_{j})_{\widetilde{a}_1,\dots,\widetilde{a}_{c+r}}^{\zeta_j}\,
=\,
a_1
\cdots
a_{c+r}\,
a_1
\cdots
a_n\,
\cdot\,
\det\,
(\widehat{\mathsf{C}}_{j})_{a_1,\dots,a_{c+r}}^{\zeta_j},
\]
which is our desired identity.

\medskip
\noindent
{\em Proof of Step 3.}
First of all, we recall the famous 

\medskip\noindent
{\bf Cramer's Rule.}
{\em \
In a commutative ring $R$, for all positive integers $N\geqslant 1$, let
$
A^0,
A^1,
\dots,
A^N
\in 
R^N
$ 
be $N+1$ column vectors, and suppose that $z_0,z_1,\dots,z_N\in R$ satisfy:
\[
A^0\,z_0
+
A^1\,z_1
+
\cdots
+
A^N\,z_N\,
=\,
\mathbf{0}.
\]
Then for all indices $0\leqslant \ell_1, \ell_2 \leqslant N$, there hold the identities:
\[
(-1)^{\ell_1}
\det\,
\big(
A^0,\dots,\widehat{A^{\ell_1}},\dots,A^N
\big)\,
z_{\ell_2}\,
=\,
(-1)^{\ell_2}
\det\,
\big(
A^0,\dots,\widehat{A^{\ell_2}},\dots,A^N
\big)\,
z_{\ell_1}.
\eqno
\qed
\]
}

In the rest of the proof, we shall view all entries of the matrices
$(\mathsf{K})_{a_1,\dots,a_{c+r}}^{\zeta_j}$,
$
(\mathsf{C})_{a_1,\dots,a_{c+r}}^{\zeta_j}
$ 
as elements in the ring
$\mathsf{H}^0(U\cap X, \mathsf{Sym}^{\bullet}\,\Omega_V^1)$.
Note that the sum of all columns of $(\mathsf{K})_{a_1,\dots,a_{c+r}}^{\zeta_j}$ vanishes:
{\footnotesize
\[
C_0\,
\Big(\frac{\zeta_0}{\zeta_j}\Big)^{\lambda_0-1}
+
\cdots
+
C_N\,
\Big(\frac{\zeta_N}{\zeta_j}\Big)^{\lambda_N-1}
\,
=\,
\mathbf{0},
\]
}
where we denote the $(i+1)$-th row of
$
(\mathsf{C})_{a_1,\dots,a_{c+r}}^{\zeta_j}
$
by
$C_i$.
Applying Cramer's rule, we receive:
{\footnotesize
\begin{equation}
\label{applying Cramer's rule}
(-1)^{\ell_1}\,
\det
(\widehat{\mathsf{C}}_{\ell_1})_{a_1,\dots,a_{c+r}}^{\zeta_j}\,
\cdot\,
\Big(\frac{\zeta_{\ell_2}}{\zeta_j}\Big)^{\lambda_{\ell_2}-1}\,
=\,
(-1)^{\ell_2}\,
\det
(\widehat{\mathsf{C}}_{\ell_2})_{a_1,\dots,a_{c+r}}^{\zeta_j}\,
\cdot\,
\Big(\frac{\zeta_{\ell_1}}{\zeta_j}\Big)^{\lambda_{\ell_1}-1}\
\in\
\mathsf{H}^0\,
\Big(
U
\cap
X,\,
\mathsf{Sym}^n\,\Omega^1_{V}
\Big)
\quad
{\scriptstyle(0\,\leqslant\,\ell_1,\,\ell_2\,\leqslant\,N)}.
\end{equation}
}

Lastly, we can check the desired identity~\thetag{\ref{show that the target form glue together}} by the following computation:
{\footnotesize
\[
\aligned
\widehat\omega_{\ell_1}\,
&
=\,
(-1)^{\ell_1}\,
a_1
\cdots
a_{c+r}
\cdot
a_1
\cdots
a_n\,
\cdot\,
\zeta_{\ell_1}^{\heartsuit_{j_1,\dots,j_n}'}\,
\cdot\,
\det\,
(\widehat{\mathsf{C}}_{\ell_1})_{a_1,\dots,a_{c+r}}^{\zeta_{\ell_1}}
\qquad
\explain{use~\thetag{\ref{definition of the target symmetric form}}}
\\
\explain{use~\thetag{\ref{applying Cramer's rule}}
for $j=\ell_1$}
\qquad
&
=\,
(-1)^{\ell_2}\,
a_1
\cdots
a_{c+r}
\cdot
a_1
\cdots
a_n\,
\cdot\,
\zeta_{\ell_1}^{\heartsuit_{j_1,\dots,j_n}'}\,
\cdot\,
\det\,
(\widehat{\mathsf{C}}_{\ell_2})_{a_1,\dots,a_{c+r}}^{\zeta_{\ell_1}}\,
\cdot\,
\Big(\frac{\zeta_{\ell_1}}{\zeta_{\ell_2}}\Big)^{\lambda_{\ell_2}-1}\,
\\
\explain{use Proposition~\ref{Observation: transition formula} below}
\qquad
&
=\,
(-1)^{{\ell_2}}\,
a_1
\cdots
a_{c+r}
\cdot
a_1
\cdots
a_n\,
\cdot\,
\zeta_{\ell_1}^{\heartsuit_{j_1,\dots,j_n}'}\,
\cdot\,
\Big(
\frac{\zeta_{\ell_2}}
{\zeta_{\ell_1}}
\Big)^{\heartsuit({\ell_2})}\,
\cdot\,
\det\,
(\widehat{\mathsf{C}}_{\ell_2})_{a_1,\dots,a_{c+r}}^{\zeta_{\ell_2}}\,
\cdot\,
\Big(\frac{\zeta_{\ell_1}}{\zeta_{\ell_2}}\Big)^{\lambda_{\ell_2}-1}
\\
\explain{$
\heartsuit(\ell_2)
=
\heartsuit_{j_1,\dots,j_n}'
+\lambda_{\ell_2}-1$}
\qquad
&
=\,
(-1)^{\ell_2}\,
a_1
\cdots
a_{c+r}
\cdot
a_1
\cdots
a_n\,
\cdot\,
\zeta_{\ell_2}^{\heartsuit_{j_1,\dots,j_n}'}\,
\cdot\,
\det\,
(\widehat{\mathsf{C}}_{\ell_2})_{a_1,\dots,a_{c+r}}^{\zeta_{\ell_2}}\,
\\
\explain{use~\thetag{\ref{definition of the target symmetric form}}}
\qquad
&
=\,
\widehat\omega_{\ell_2}
\qquad
\explain{\smiley}.
\endaligned
\]
}
Thus we finish the proof.
\end{proof}

An essential ingredient in the above proof is to compare the same determinant in different trivializations $\zeta_{\ell_1}$, $\zeta_{\ell_2}$.
Now we give general transition formulas.

\begin{Proposition}
\label{Observation: transition formula}
For all $0\leqslant j,\,\ell_1,\, \ell_2 \leqslant N$,
for any Zariski open set 
$U\subset \mathsf{D}(\zeta_{\ell_1})\cap\mathsf{D}(\zeta_{\ell_2})$
with trivializations 
$
\mathcal{A}_1\big{\vert}_U
=
\mathcal{O}_U
\cdot
a_1$,
$\dots$,
$
\mathcal{A}_{c+r}\big{\vert}_U
=
\mathcal{O}_U
\cdot
a_{a+r}$,
there hold the transition formulas:
\begin{equation}
\label{nice transition identity, two coordinates}
\det\,
(\widehat{\mathsf{C}}_{j})_{a_1,\dots,a_{c+r}}^{\zeta_{\ell_1}}\,
=\,
\Big(
\frac{\zeta_{\ell_2}}
{\zeta_{\ell_1}}
\Big)^{\heartsuit(j)}\,
\cdot\,
\det\,
(\widehat{\mathsf{C}}_{j})_{a_1,\dots,a_{c+r}}^{\zeta_{\ell_2}}\
\in\
\mathsf{H}^0\,
\Big(
U,\,
\mathsf{Sym}^n\,\Omega^1_{\mathbf{P}}\Big),
\end{equation}
with 
$
\heartsuit(j)
=
\heartsuit_{j_1,\dots,j_n}'
+\lambda_{j}-1
$.
\end{Proposition}

\begin{proof}
Our idea is to expand the two determinants and to compare  each pair of corresponding terms.
Without loss of generality, we may assume $j=0$.

For $i=1\cdots N$, $k=1\cdots N$, we denote
the $(i, k)$-th entry of $(\widehat{\mathsf{C}}_{0})_{a_1,\dots,a_{c+r}}^{\zeta_{\ell_1}}$ 
(resp. $(\widehat{\mathsf{C}}_{0})_{a_1,\dots,a_{c+r}}^{\zeta_{\ell_2}}$)
by $c_{i,j}^1$ (resp. $c_{i,j}^2$).
First of all, we recall all the entries:
\[
\footnotesize
\aligned
c_{p,k}^{\delta}\,
:=\,
\frac{A_p^k}{a_p\cdot \zeta_{\ell_\delta}^{\epsilon_p^k}}\,
\cdot\,
\frac{\zeta_k}
{\zeta_{\ell_\delta}},
\qquad
c_{c+r+q,k}^{\delta}\,
:=\,
\mathrm{d}\,
\bigg(
\frac{A_q^k}{a_q\cdot \zeta_{\ell_{\delta}}^{\epsilon_q^k}}
\bigg)\,
\cdot\,
\frac{\zeta_k}
{\zeta_{\ell_{\delta}}}
+
\lambda_k\,
\frac{A_q^k}{a_q\cdot \zeta_{\ell_{\delta}}^{\epsilon_q^k}}\,
\cdot\,
\mathrm{d}\,
\Big(
\frac{\zeta_k}
{\zeta_{\ell_{\delta}}}
\Big)
\qquad
{\scriptstyle(\delta\,=\,1,\,2;\,\,  
p\,=\,1\,\cdots\, c+r;\,\, q\,=\,1\,\cdots\, n;\,\,
k\,=\,1\,\cdots\,N)}.
\endaligned
\]
By much the same reasoning as in~\thetag{\ref{compute something}},
 we can obtain the transition formulas:
\[
\aligned
c_{p,k}^{1}\,
&
=\,
c_{p,k}^{2}\,
\cdot\,
(\zeta_{\ell_2}/\zeta_{\ell_1})^{\epsilon_p^k+1},
\\
c_{c+r+q,k}^{1}\,
&
=\,
c_{c+r+q,k}^{2}\,
\cdot\,
(\zeta_{\ell_2}/\zeta_{\ell_1})^{\epsilon_q^k+1}\,
+\,
c_{q,k}^{2}\,
\cdot\,
(\epsilon_q^k+\lambda_k)\,
(\zeta_{\ell_2}/\zeta_{\ell_1})^{\epsilon_q^k}\,
\mathrm{d}\,
(
\zeta_{\ell_2}/\zeta_{\ell_1}
).
\endaligned
\]
Recalling that $\epsilon_p^k+\lambda_k=d_p$,
we thus rewrite the above identities as:
\begin{equation}
\label{transition identities between two matrices in each entry}
\aligned
c_{p,k}^{1}\,
&
=\,
c_{p,k}^{2}\,
\cdot\,
(\zeta_{\ell_2}/\zeta_{\ell_1})^{d_p-(\lambda_k-1)},
\\
c_{c+r+q,k}^{1}\,
&
=\,
c_{c+r+q,k}^{2}\,
\cdot\,
(\zeta_{\ell_2}/\zeta_{\ell_1})^{d_q-(\lambda_k-1)}\,
+\,
\underline{
c_{q,k}^{2}\,
\cdot\,
d_q\,
(\zeta_{\ell_2}/\zeta_{\ell_1})^{d_q-\lambda_k}\,
\mathrm{d}\,
(
\zeta_{\ell_2}/\zeta_{\ell_1}
)
}.
\endaligned
\end{equation}

Now, comparing~\thetag{\ref{transition identities between two matrices in each entry}} with the desired formula~\thetag{\ref{nice transition identity, two coordinates}}, we may anticipate that, the underlined terms would bring some trouble,
since no terms
$
\mathrm{d}\,
(
\zeta_{\ell_2}/\zeta_{\ell_1}
)
$ 
appear on the right-hand-side of~\thetag{\ref{nice transition identity, two coordinates}}.
Nevertheless, we can overcome this difficulty firstly by 
observing:
\begin{equation}
\label{central observation of Prop. 4.3}
\begin{vmatrix}
c_{q,k_1}^{1}
&
c_{q,k_2}^{1}
\\
c_{c+r+q,k_1}^{1}
&
c_{c+r+q,k_2}^{1}
\end{vmatrix}\,
=\,
\begin{vmatrix}
c_{q,k_1}^{2}
&
c_{q,k_2}^{2}
\\
c_{c+r+q,k_1}^{2}
&
c_{c+r+q,k_2}^{2}
\end{vmatrix}\,
\cdot\,
(\zeta_{\ell_2}/\zeta_{\ell_1})^{d_q-(\lambda_{k_1}-1)}\,
\cdot\,
(\zeta_{\ell_2}/\zeta_{\ell_1})^{d_{q}-(\lambda_{k_2}-1)}
\qquad
{\scriptstyle( 
q\,=\,1\,\cdots\, n;\,\,
k_1,\,k_2\,=\,1\,\cdots\,N)},
\end{equation}
and secondly by using a tricky Laplace expansion of the determinant:
\begin{equation}
\label{Laplace expansion of Prop 4.3}
\det\,
(\widehat{\mathsf{C}}_{0})_{a_1,\dots,a_{c+r}}^{\zeta_{\ell_1}}\,
=\,
\sum\,
\mathsf{Sign}(\pm)\,
\cdot
\prod_{q=1}^n\,
\begin{vmatrix}
c_{q,k_q^1}^{1}
&
c_{q,k_q^2}^{1}
\\
c_{c+r+q,k_q^1}^{1}
&
c_{c+r+q,k_q^2}^{1}
\end{vmatrix}\,
\cdot\,
\prod_{p=n+1}^{c+r}\,
c_{p,k_p}^{1},
\end{equation}
where the sum runs through all choices of $N=2n+(c+r-n)$ indices
$k_1^1<k_1^2$, $\dots$, $k_n^1<k_n^2$, $k_{n+1},\dots,k_{c+r}$ such that their union is exactly
$\{1,\dots,N\}$, and where
$\mathsf{Sign}(\pm)$ is either $1$ or $-1$
uniquely determined by the choices of indices. 
Now, using~\thetag{\ref{central observation of Prop. 4.3}}, we see that each term in~\thetag{\ref{Laplace expansion of Prop 4.3}} is equal 
to:
\[
\mathsf{Sign}(\pm)\,
\cdot
\prod_{q=1}^n\,
\begin{vmatrix}
c_{q,k_q^1}^{2}
&
c_{q,k_q^2}^{2}
\\
c_{c+r+q,k_q^1}^{2}
&
c_{c+r+q,k_q^2}^{2}
\end{vmatrix}\,
\cdot\,
\prod_{p=n+1}^{c+r}\,
c_{p,k_p}^{2}
\]
multiplied by
$(\zeta_{\ell_2}/\zeta_{\ell_1})^{\heartsuit}$,
where:
\[
\aligned
\heartsuit\,
&
:=\,
\sum_{q=1}^n\,
\Big[
d_q-(\lambda_{k_q}^1-1)\,
+\,
d_{q}-(\lambda_{k_q}^2-1)
\Big]\,
+\,
\sum_{p=n+1}^{c+r}\,
\big(
d_q-(\lambda_{k_p}-1)
\big)
\\
&
\,\,=\,
\sum_{p=1}^{c+r}\,
d_p\,
+\,
\sum_{q=1}^{n}\,
d_q\,
-\,
\sum_{k=1}^{N}\,
(\lambda_k-1)\,
\\
&
\,\,=\,
\heartsuit_{j_1,\dots,j_n}'
+\lambda_{0}-1.
\endaligned
\]
Thus~\thetag{\ref{Laplace expansion of Prop 4.3}}
factors as:
\[
\aligned
\det\,
(\widehat{\mathsf{C}}_{0})_{a_1,\dots,a_{c+r}}^{\zeta_{\ell_1}}\,
&
=\,
(\zeta_{\ell_2}/\zeta_{\ell_1})^{\heartsuit}\,
\cdot\,
\sum\,
\mathsf{Sign}(\pm)\,
\cdot
\prod_{q=1}^n\,
\begin{vmatrix}
c_{q,k_1^q}^{2}
&
c_{q,k_2^q}^{2}
\\
c_{c+r+q,k_1^q}^{2}
&
c_{c+r+q,k_2^q}^{2}
\end{vmatrix}\,
\cdot\,
\prod_{p=n+1}^{c+r}\,
c_{p,k_p}^{2}
\\
\explain{use Laplace expansion again}
\qquad
&
=\,
(\zeta_{\ell_2}/\zeta_{\ell_1})^{\heartsuit}\,
\cdot\,
\det\,
(\widehat{\mathsf{C}}_{0})_{a_1,\dots,a_{c+r}}^{\zeta_{\ell_2}},
\endaligned
\]
whence we conclude the proof.
\end{proof}

\section{\bf `Hidden' Symmetric Differential Forms}
\label{section: Hidden Symmetric Differential Forms}
Comparing the two approaches
in~\cite[Section 6]{Xie-2015-arxiv},
the scheme-theoretic one has the advantage in further generalizations, while
the geometric one is superior in discovering
the `hidden' symmetric differential forms~~\cite[Proposition 6.12]{Xie-2015-arxiv}.
Skipping the thinking process, we present the corresponding generalizations of these symmetric forms
as follows. 

We assume that $\lambda_0,\dots,\lambda_N\geqslant 2$
in this section.
For any 
$\eta=1\cdots n-1$, for any indices
$
0\leqslant v_1<\dots<v_{\eta}\leqslant N
$
and
$
1\leqslant j_1 < \cdots < j_{n-\eta} \leqslant c
$,
write 
$
\{0,\dots,N\}
\setminus 
\{v_1,\dots,v_{\eta}\}
$ 
in the ascending order
$
r_0<r_1<\cdots<r_{N-\eta}
$, and then
denote by ${}_{v_1,\dots,v_\eta}{\sf M}_{j_1,\dots,j_{n-\eta}}$ 
the
$(N-\eta)\times (N-\eta+1)$ submatrix of ${\sf M}$
 determined by the first $c+r$ rows and the selected rows
$c+r+j_1,\dots,c+r+j_{n-\eta}$ as well as
 the $(N-\eta+1)$ columns $r_0+1,\dots,r_{N-\eta}+1$.
Next, for every index $j\in \{0,\dots,N\}\setminus \{v_1,\dots,v_{\eta}\}$, let
${}_{v_1,\dots,v_{\eta}}\widehat{{\sf M}}_{j_1,\dots,j_{n-\eta};\,j}$ 
denote the submatrix of ${}_{v_1,\dots,v_{\eta}}{\sf M}_{j_1,\dots,j_{n-\eta}}$ 
obtained by deleting the column which is originally contained in the
$(j+1)$-th column of ${\sf M}$.
 Lastly, denote by ${}_{v_1,\dots,v_{\eta}}\mathbf{P}\subset\mathbf{P}$ 
the subvariety defined by sections $\zeta_{v_1},\dots,\zeta_{v_\eta}$ (`vanishing coordinates'), 
and denote ${}_{v_1,\dots,v_{\eta}}X:=X\cap {}_{v_1,\dots,v_{\eta}}\mathbf{P}$.
Setting:
\[
\mathcal{A}_{j_1,\dots,j_{n-\eta}}^{1,\dots,c+r}\,
:=\,
\mathcal{A}_{1}
\otimes
\cdots
\otimes
\mathcal{A}_{c+r}
\otimes
\mathcal{A}_{j_1}
\otimes
\cdots
\otimes
\mathcal{A}_{n-\eta},
\]
by much the same reasoning as in Proposition~\ref{corollary of naive forms}, 
we have

\begin{Proposition}
\label{general naive determinant with vanishing coordinates is well defined}
For all $j=0\cdots N-\eta$, the following $N+1-\eta$
sections: 
\[
\aligned\,
{}_{v_1,\dots,v_{\eta}}\psi_{j_1,\dots,j_{n-\eta};\,r_j}\,
&
:=\,
(-1)^j
\det
\big(
{}_{v_1,\dots,v_{\eta}}\widehat{{\sf M}}_{j_1,\dots,j_{n-\eta};\,r_j}
\big)\\
&\,\,
\in\
\mathsf{H}^0\,
\Big(
{}_{v_1,\dots,v_{\eta}}\mathbf{P},\,\mathsf{Sym}^{n-\eta}\,\Omega^1_{\mathbf{P}}
\otimes
\mathcal{A}_{j_1,\dots,j_{n-\eta}}^{1,\dots,c+r}
\otimes 
\mathcal{L}^{\heartsuit_{j_1,\dots,j_{n-\eta}}}
\Big),
\endaligned
\]
when restricted to ${}_{v_1,\dots,v_{\eta}}X$, give one and the same  symmetric differential form:
\[
{}_{v_1,\dots,v_{\eta}}\psi_{j_1,\dots,j_{n-\eta}}\
\in\
\mathsf{H}^0\,
\Big(
{}_{v_1,\dots,v_{\eta}}X,\,
\mathsf{Sym}^{n-\eta}\,\Omega^1_{V}
\otimes
\mathcal{A}_{j_1,\dots,j_{n-\eta}}^{1,\dots,c+r}
\otimes 
\mathcal{L}^{\heartsuit_{j_1,\dots,j_{n-\eta}}}
\Big),
\]
with the twisted degree:
\[
\heartsuit_{j_1,\dots,j_{n-\eta}}\,
=\,
\sum_{p=1}^{c+r}\,d_{p}
+
\sum_{q=1}^{n-\eta}\,d_{j_q}.
\eqno
\qed
\]
\end{Proposition}

Moreover, playing the dividing trick again, we obtain an analogue of Proposition~\ref{the symmetric differential forms with many very ample line bundle}.

\begin{Proposition}
\label{the hidden general symmetric differential forms}
For all $j=0\cdots N-\eta$,
the formal symmetric differential forms:
\[
{}_{v_1,\dots,v_{\eta}}\widehat\omega_{j_1,\dots,j_{n-\eta};\,r_j}\,
=\,
\frac
{
(-1)^j
\det
\big(
{}_{v_1,\dots,v_{\eta}}\widehat{{\sf M}}_{j_1,\dots,j_{n-\eta};\,r_j}
\big)
}
{
\zeta_{r_0}^{\lambda_{r_0}-1}
\cdots
\zeta_{r_{N-\eta}}^{\lambda_{r_{N-\eta}}-1}
}
\]
are
well-defined
sections in:
\[
\mathsf{H}^0\,
\Big(
\mathsf{D}(\zeta_{r_j})\cap
{}_{v_1,\dots,v_{\eta}}\mathbf{P},\,\mathsf{Sym}^{n-\eta}\,\Omega^1_{\mathbf{P}}
\otimes
\mathcal{A}_{j_1,\dots,j_{n-\eta}}^{1,\dots,c+r}
\otimes 
\mathcal{L}^{{}_{v_1,\dots,v_{\eta}}\heartsuit_{j_1,\dots,j_{n-\eta}}'}
\Big),
\]
with the twisted degree:
\[
{}_{v_1,\dots,v_{\eta}}\heartsuit_{j_1,\dots,j_{n-\eta}}'\,
:=\,
\sum_{p=1}^{c+r}\,d_{p}
+
\sum_{q=1}^{n-\eta}\,d_{j_q}
-
\sum_{k=0}^{N}\,(\lambda_k-1)
+
\sum_{\mu=1}^{\eta}\,(\lambda_{v_{\mu}}-1).
\]
Moreover, when restricted to ${}_{v_1,\dots,v_{\eta}}X$, they
glue together to make a global section:
\[
{}_{v_1,\dots,v_{\eta}}\omega_{j_1,\dots,j_{n-\eta}}\
\in\
\mathsf{H}^0\,
\Big(
{}_{v_1,\dots,v_{\eta}}X,\,
\mathsf{Sym}^{n-\eta}\,\Omega^1_{V}
\otimes
\mathcal{A}_{j_1,\dots,j_{n-\eta}}^{1,\dots,c+r}
\otimes 
\mathcal{L}^{{}_{v_1,\dots,v_{\eta}}\heartsuit_{j_1,\dots,j_{n-\eta}}'}
\Big).
\eqno
\qed
\]
\end{Proposition}

\section{\bf Applications of MCM}
\label{Section: Applications of MCM}
\subsection{Motivation}
Recall~\cite[Section 7]{Xie-2015-arxiv} that the moving coefficients method is devised to produce
as many negatively twisted symmetric differential forms as possible, by  manipulating the determinantal structure of the constructed symmetric differential forms. 
Since Propositions~\ref{the symmetric differential forms with many very ample line bundle}, \ref{the hidden general symmetric differential forms} 
exactly share the same determinantal shape,
it is possible to adapt MCM for the aim of
Theorem~\ref{Gentle generalization of Debarre conjecture}, which coincides with Theorem~\ref{Debarre Ampleness Conjecture} in the case that 
$\mathbf{P}=\mathbb{P}_{\mathbb{K}}^N$, $\mathcal{L}=\mathcal{O}_{\mathbb{P}_{\mathbb{K}}^N}(1)$.
Indeed, by introducing $c$ auxiliary line bundles $\mathcal{A}_1,\dots, \mathcal{A}_{c}\approx$ trivial line bundle, we can even treat the case of $c$  ample line bundles $\mathcal{L}+\mathcal{A}_1$, $\dots$, $\mathcal{L}+\mathcal{A}_c\approx \mathcal{L}$, and eventually we will obtain Theorem~\ref{Main Theorem, 2nd}.

\subsection{Adaptation}
\label{subsection: Adaptation}
Let $\mathbf{P}$ be a smooth projective
$\mathbb{K}$-variety of dimension $N$, equipped with a very ample line bundle $\mathcal{L}$.
By Bertini's theorem, we may choose
$N+1$ simple normal crossing global sections $\zeta_0,\dots,\zeta_N$ of $\mathcal{L}$,
and we shall view them as the `homogeneous coordinates'
of $\mathbf{P}$. Thus, we may
`identify' 
$(\mathbf{P}, \mathcal{L})$ with $\big(\mathbb{P}_{\mathbb{K}}^N, \mathcal{O}_{\mathbb{P}_{\mathbb{K}}^N}(1)\big)$
in the sense that locally they have the same coordinates $[\xi_0:\cdots:\zeta_N]\approx [z_0:\cdots:z_N]$, and therefore we can generalize local computations of the later one to the former one,
like what we perform in Section~\ref{section: A Dividing Trick}.
This treatment is also visible in~\cite{Brotbek-Lionel-2015}.

Let $c\geqslant 1$, $r\geqslant 0$ be two integers with $2c+r\geqslant N$ and $c+r< N$, and  
let $\mathcal{A}_1,\dots,\mathcal{A}_{c+r}$ be
$c+r$ auxiliary line bundles to be determined.
Now, we start to adapt the machinery of MCM.
First of all, introduce the following $c+r$ `flexible' sections which copy the major ingredient {\bf (3)}:
\begin{equation}
\label{F_i-moving-coefficient-method-full-strenghth}
\aligned
F_i\,
=\,
\sum_{j=0}^N\,
A_i^j\,
\zeta_j^{d}
\,
+
\,
\sum_{l=c+r+1}^{N}\,
\sum_{0\leqslant j_0<\cdots<j_l\leqslant N}\,
\sum_{k=0}^l\,
&
M_i^{j_0,\dots,j_l;j_k}\,
\zeta_{j_0}^{\mu_{l,k}}
\cdots
\widehat{\zeta_{j_k}^{\mu_{l,k}}}
\cdots
\zeta_{j_l}^{\mu_{l,k}}
\zeta_{j_k}^{d-l\mu_{l,k}}
\\
&
\in\
\mathsf{H}^0\,(\mathbf{P},\,\mathcal{A}_i\otimes \mathcal{L}^{\epsilon_i}\otimes \mathcal{L}^{d})
\qquad
{\scriptstyle(i\,=\,1\,\cdots\,c+r)},
\endaligned
\end{equation}
where all coefficients $A_i^{\bullet}, M_i^{\bullet;\bullet}$ are some global sections of $\mathcal{A}_i\otimes \mathcal{L}^{\epsilon_i}$
for some fixed integers $\epsilon_i\geqslant 1$, and where all positive integers $\mu^{l,k}, d$ are to be chosen by a certain {\sl Algorithm}, which is designed to make all the symmetric differential forms obtained later have {\sl negative twist}.
For better comprehension, we will make the Algorithm clear in Subsection~\ref{Subsection: Delayed Algorithm} below,
and for the time being we roughly summarize it as:
\begin{equation}
\label{rough algorithm}
1
\leqslant
\max\,\{\epsilon_i\}_{i=1\cdots c+r}
\ll
\underbrace{
\mu_{c+r+1,0}
\ll
\cdots
\ll
\mu_{c+r+1,c+r+1}
}_{
\mu_{c+r+1,\bullet}
\text{ grow exponentially}
}
\ll
\cdots
\cdots
\ll
\underbrace{
\mu_{N,0}
\ll
\cdots
\ll
\mu_{N,N}
}_{
\mu_{N,\bullet}
\text{ grow exponentially}
}
\ll
d.
\end{equation}

Let $V\subset \mathbf{P}$ be the subvariety defined by
the first $c$ sections $F_1,\dots,F_c$, and let
$X\subset \mathbf{P}$ be the subvariety defined by
all the $c+r$ sections $F_1,\dots,F_{c+r}$.
{\em A priori}, we require all the line bundles
$\mathcal{A}_i\otimes \mathcal{L}^{\epsilon_i}$
to be very ample, so that 
for generic choices of parameters: 
\[
A_i^{\bullet}, 
M_i^{\bullet;\bullet}\
\in\
\mathsf{H}^0\,(\mathbf{P},\,\mathcal{A}_i\otimes \mathcal{L}^{\epsilon_i})
\qquad
{\scriptstyle{(i\,=\,1\,\cdots\,c+r)}},
\] 
both intersections $V$, $X$ are smooth complete (the proof is much the same as that of Bertini's Theorem,
see Subsection~\ref{subsection: Bertini-type assertions}).

\subsection{Manipulations}
Now, 
we apply MCM to construct a series of negatively twisted symmetric differential forms. 
For shortness, we will refer  to~\cite[Section~7]{Xie-2015-arxiv} for skipped details, in which the canonical setting
$\{\mathbb{P}_{\mathbb{K}}^N,
\mathcal{O}_{\mathbb{P}_{\mathbb{K}}^N}(1), (z_0,\dots,z_N)\}$ there plays the same role as that of
 $\{\mathbf{P},\mathcal{L}, (\zeta_0,\dots,\zeta_N)\}$ in our treatment here. 

\smallskip
To begin with, we rewrite each section $F_i$ in~\thetag{\ref{F_i-moving-coefficient-method-full-strenghth}} as (cf.~\cite[p.~43, (104)]{Xie-2015-arxiv}):
{\footnotesize
\begin{equation}
\label{F_i = the sum of 2N+2 sections}
F_i\,
=\,
\sum_{j=0}^N\,
\underbrace{
\Big(
A_i^j\,\zeta_j^{d}
+
\sum_{l=c+r+1}^{N-1}\,
\sum_{
\substack
{0\leqslant j_0<\cdots<j_l\leqslant N
\\
j_k=j\,\text{for some}\,0\leqslant k \leqslant l
}}\,
M_i^{j_0,\dots,j_l;j_k}\,
\zeta_{j_0}^{\mu_{l,k}}
\cdots
\widehat{\zeta_{j_k}^{\mu_{l,k}}}
\cdots
\zeta_{j_l}^{\mu_{l,k}}
\zeta_{j_k}^{d-l\mu_{l,k}}
\Big)
}_{
\text{for each }j\,=\,0\,\cdots\, N,
\text{ we view this whole bracket as one section}
}\,
+\,
\sum_{k=0}^N\,
\underbrace{
M_i^{0,\dots,N;k}\,
\zeta_{0}^{\mu_{N,k}}
\cdots
\widehat{\zeta_{k}^{\mu_{N,k}}}
\cdots
\zeta_{N}^{\mu_{N,k}}
\zeta_{k}^{d-N\,\mu_{N,k}}
}_{
\text{for each }k\,=\,0\,\cdots\, N,
\text{ we view it as one section}
}.
\end{equation}
}\!\!
Thus, we view each $F_i$ as 
the sum of $2N+2=\sum_{j=0}^N\,1+\sum_{k=0}^N\,1$ sections of the same line bundle,
as indicate above.

Next, we
construct a $(c+r+c)\times (2N+2)$ formal matrix $\mathsf{M}$ such that, for every $i=1\cdots c+r$, $j=1\cdots c$, its $i$-th row copies
the $2N+2$ sections in the sum of $F_i$
in the exact order, and its $(c+r+j)$-th row is the formal differential of the $j$-th row. 

Write the $2N+2$ columns of $\sf M$ as:
\begin{equation}
\label{M=(...)}
{\sf M}
=
\Big(
{\sf A}_0
\mid
\cdots
\mid
{\sf A}_N
\mid
{\sf B}_0
\mid
\cdots
\mid
{\sf B}_N
\Big).
\end{equation}
For every $\nu=0\cdots N$, we construct the matrix:
\begin{equation}
\label{K^nu=(...)}
{\sf K}^{\nu}\,
:=\,
\Big(
\mathsf{A}_0
\mid
\cdots
\mid
\widehat{\mathsf{A}_{\nu}}
\mid
\cdots
\mid
\mathsf{A}_N
\mid
\mathsf{A}_{\nu}+\sum_{j=0}^{N}\mathsf{B}_{j}
\Big),
\end{equation}
where the last column is understood to appear in the
`omitted' column.
Also, for every $\tau=0\cdots N-1$ and every $\rho=\tau+1\cdots N$, we construct the matrix: 
\begin{equation}
\label{K^(tau,rho)=(...)}
{\sf K}^{\tau,\,\rho}\,
:=\,
\Big(
\mathsf{A}_0+\mathsf{B}_{0}
\mid
\cdots
\mid
\mathsf{A}_{\tau}+\mathsf{B}_{\tau}
\mid
\mathsf{A}_{\tau+1}
\mid
\cdots
\mid
\widehat{\mathsf{A}_{\rho}}
\mid
\cdots
\mid
\mathsf{A}_{N}
\mid
\mathsf{A}_{\rho}+\sum_{j=\tau+1}^N\mathsf{B}_{j}
\Big).
\end{equation}

Now, fix a positive integer $\varheartsuit\geqslant 1$
such that:
\begin{equation}
\label{heart should be big enough}
\mathcal{A}_{i}
\otimes 
\mathcal{L}^{-\varheartsuit}\,
<\,
0
\qquad
{\scriptstyle(i\,=\,1\,\cdots\, c+r)}.
\end{equation}
Recalling the rough Algorithm~\thetag{\ref{rough algorithm}},
observe in~\thetag{\ref{F_i = the sum of 2N+2 sections}}, \thetag{\ref{K^nu=(...)}}
that the $N+1$ columns of ${\sf K}^{\nu}$
are subsequently divisible by:
\[
\zeta_0^{d-\delta_N},
\dots, 
\widehat{\zeta_\nu^{d-\delta_N}},
\dots,
\zeta_N^{d-\delta_N},
\zeta_\nu^{\mu_{N,0}},
\] 
where $\delta_N:=(N-1)\,\mu_{N-1,N-1}$.
Thus, 
applying Proposition~\ref{the symmetric differential forms with many very ample line bundle}, 
for every $1 \leqslant j_1 <
\dots < j_n \leqslant c$,
we obtain a global symmetric differential form:
\[
\phi_{j_1,\dots,j_n}^{\nu}
\
\in\
\mathsf{H}^0
\big(
X,\mathsf{Sym}^{n}\,\Omega_V
\otimes
\underbrace{
\mathcal{A}_{j_1,\dots,j_n}^{1,\dots,c+r}
\otimes 
\mathcal{L}^{\heartsuit_{j_1,\dots,j_n}^{\nu}}
}_{
<\,0, \text{ because of}~\thetag{\ref{heart should be big enough}},
\thetag{\ref{value of heart, nu}}
}
\big)
\qquad
{\scriptstyle (\nu\,=\,0\,\cdots\, N)},
\]
with negative twist: 
{\footnotesize
\begin{equation}
\label{value of heart, nu}
\aligned
\heartsuit_{j_1,\dots,j_n}^{\nu}\,
&
=\,
\sum_{p=1}^{c+r}\,(d+\epsilon_{p})
+
\sum_{q=1}^{n}\,(d+\epsilon_{j_q})
-
\sum_{j=0,
j\neq \nu}^{N}\,(d-\delta_N-1)
-
(\mu_{N,0}-1)
\\
&
=\,
-\,\mu_{N,0}
+
N\,\delta_N
+
\sum_{p=1}^{c+r}\,
\epsilon_p
+
\sum_{q=1}^n\,
\epsilon_{j_q}
+
N
+
1
\\
\explain{by Algorithm~\thetag{\ref{rough algorithm}}}
\qquad
&
\leqslant\,
-\,N\,\varheartsuit.
\endaligned
\end{equation}
}\!\!
Similarly, observe that
the $N+1$ columns of ${\sf K}^{\tau,\,\rho}$
are subsequently divisible by:
\[
\zeta_0^{d-N\,\mu_{N,0}},
\dots, 
\zeta_{\tau}^{d-N\,\mu_{N,\tau}},
\zeta_{\tau+1}^{d-\delta_{N}},
\dots,
\widehat{\zeta_\nu^{d-\delta_N}},
\dots,
\zeta_N^{d-\delta_N},
\zeta_\nu^{\mu_{N,\tau+1}},
\]
thus by Proposition~\ref{the symmetric differential forms with many very ample line bundle} we obtain:
\[
\psi_{j_1,\dots,j_n}^{\tau,\,\rho}\
\in\
\mathsf{H}^0
\big(
X,\mathsf{Sym}^{n}\,\Omega_V
\otimes
\underbrace{
\mathcal{A}_{j_1,\dots,j_n}^{1,\dots,c+r}
\otimes 
\mathcal{L}^{\heartsuit_{j_1,\dots,j_n}^{\tau,\,\rho}}
}_{
<\,0, \text{ because of}~\thetag{\ref{heart should be big enough}},
\thetag{\ref{value of heart, tau-rho}}
}
\big)
\qquad
{\scriptstyle (\tau\,=\,0\,\cdots\, N-1,\,\,
\rho\,=\,\tau+1\,\cdots\, N)},
\]
with negative twist: 
{\footnotesize
\begin{equation}
\label{value of heart, tau-rho}
\aligned
\heartsuit_{j_1,\dots,j_n}^{\tau,\,\rho}\,
&
=\,
\sum_{p=1}^{c+r}\,(d+\epsilon_{p})
+
\sum_{q=1}^{n}\,(d+\epsilon_{j_q})
-
\sum_{k=0}^{\tau}\,(d-N\,\mu_{N,k}-1)
-
\sum_{j=\tau+1,j\neq \rho}^{N}\,(d-\delta_N-1)
-
(\mu_{N,\tau+1}-1)
\\
&
=\,
-\,\mu_{N,\tau+1}
+
\sum_{k=0}^{\tau}N\,\mu_{N,k}
+
(N-\tau-1)\,\delta_N
+
\sum_{p=1}^{c+r}\,
\epsilon_p
+
\sum_{q=1}^n\,
\epsilon_{j_q}
+
N
+
1
\\
\explain{by Algorithm~\thetag{\ref{rough algorithm}}}
\qquad
&
\leqslant\,
-\,N\,\varheartsuit.
\endaligned
\end{equation}
}

Recalling the notation in~Section~\ref{section: Hidden Symmetric Differential Forms},
for any 
$\eta=1\cdots n-1$, for any `vanishing' indices
$
0\leqslant v_1<\dots<v_{\eta}\leqslant N
$,
by
applying Proposition~\ref{the hidden general symmetric differential forms},
we can construct a series of negatively twisted symmetric differential forms 
over the `coordinates vanishing part':
\[
\omega_{\ell}\
\in
\Gamma
\big(
{}_{v_1,\dots,v_{\eta}}X,\mathsf{Sym}^{n-\eta}\,\Omega_V
\otimes
\text{negative twist}
\big)
\qquad
{\scriptstyle
(\,\ell\,=\,1,\,2,\,\dots).
}
\]
The procedure is much the same as before.
First, we rewrite each section $F_i$ in~\thetag{\ref{F_i-moving-coefficient-method-full-strenghth}} as:
{\footnotesize
\[
F_i
=
\sum_{j=0}^{N-\eta}\,
A_i^{r_j}\,
\zeta_{r_j}^{d}
+
\sum_{l=c+r+1}^{N-\eta}\,
\sum_{0\leqslant j_0<\cdots<j_l\leqslant N-\eta}\,
\sum_{k=0}^l\,
M_i^{r_{j_0},\dots,r_{j_l};r_{j_k}}\,
\zeta_{r_{j_0}}^{\mu_{l,k}}
\cdots
\widehat{\zeta_{r_{j_k}}^{\mu_{l,k}}}
\cdots
\zeta_{r_{j_l}}^{\mu_{l,k}}
\zeta_{r_{j_k}}^{d-l\mu_{l,k}}
+ 
\mathsf{negligible\,\,terms},
\]
}\!\!
so that $F_i$ has the same structure as~\thetag{\ref{F_i-moving-coefficient-method-full-strenghth}}, in the sense of  replacing:
\[
N \leftrightarrow N-\eta,
\qquad
\{0,\dots,N\}\
\longleftrightarrow\
\{
r_0,
\dots,
r_{N-\eta}
\}.
\] 
Thus, we can repeat the above
 manipulations. For shortness, we 
skip all details (cf.~\cite[Subsection 7.3]{Xie-2015-arxiv}) and only
state the results. 

For every 
$
1 
\leqslant 
j_1 
< 
\cdots 
< 
j_{n-\eta}
\leqslant 
c
$,
for every $\nu=0\cdots N-\eta$,
we obtain a symmetric differential form:
\[
{}_{v_1,\dots,v_{\eta}}\phi_{j_1,\dots,j_{n-\eta}}^{\nu}
\in
\Gamma
\big(
{}_{v_1,\dots,v_{\eta}}X,\mathsf{Sym}^{n-\eta}\,\Omega_V
\otimes
\underbrace{
\mathcal{A}_{j_1,\dots,j_{n-\eta}}^{1,\dots,c+r}
\otimes 
\mathcal{L}^{{}_{v_1,\dots,v_{\eta}}\heartsuit_{j_1,\dots,j_{n-\eta}}^{\nu}}
}_{
<\,0, \text{ because of }
\thetag{\ref{value of heart, nu, 2}}
}
\big),
\]
with negative twist (set $
\delta_{N-\eta}
:=
(N-\eta-1)\,\mu_{N-\eta-1,N-\eta-1}
$):  
\begin{equation}
\label{value of heart, nu, 2}
{}_{v_1,\dots,v_{\eta}}\heartsuit_{j_1,\dots,j_{n-\eta}}^{\nu}\,
=\,
-\,\mu_{N-\eta,0}
+
(N-\eta)\,\delta_{N-\eta}
+
\sum_{i=1}^{c+r}\,
\epsilon_i
+
\sum_{\ell=1}^{n-\eta}\,
\epsilon_{j_\ell}
+
(N-\eta)
+
1\
\leqslant\
-\,(N-\eta)\,\varheartsuit.
\end{equation}
Also, for every $\tau=0\cdots N-\eta-1$ and every $\rho=\tau+1\cdots N-\eta$, we obtain:
\[
{}_{v_1,\dots,v_{\eta}}\phi_{j_1,\dots,j_{n-\eta}}^{\tau,\,\rho}
\in
\Gamma
\big(
{}_{v_1,\dots,v_{\eta}}X,\mathsf{Sym}^{n-\eta}\,\Omega_V
\otimes
\underbrace{
\mathcal{A}_{j_1,\dots,j_{n-\eta}}^{1,\dots,c+r}
\otimes 
\mathcal{L}^{{}_{v_1,\dots,v_{\eta}}\heartsuit_{j_1,\dots,j_{n-\eta}}^{\tau,\,\rho}}
}_{
<\,0, \text{ because of }
\thetag{\ref{negative twisted degree, tau, rho, vanishing coordinates}}
}
\big),
\]
with negative twist:  
{\footnotesize
\begin{equation}
\label{negative twisted degree, tau, rho, vanishing coordinates}
{}_{v_1,\dots,v_{\eta}}\heartsuit_{j_1,\dots,j_{n-\eta}}^{\tau,\,\rho}\,
=\,
-\,\mu_{N-\eta,\tau+1}
+
\sum_{k=0}^{\tau}(N-\eta)\,\mu_{N-\eta,k}
+
(N-\eta-\tau-1)\,\delta_{N-\eta}
+
\sum_{i=1}^{c+r}\,
\epsilon_i
+
\sum_{\ell=1}^{n-\eta}\,
\epsilon_{j_\ell}
+
(N-\eta)
+
1\
\leqslant\
-\,(N-\eta)\,\varheartsuit.
\end{equation}
}

\subsection{A Natural Algorithm}
\label{Subsection: Delayed Algorithm}
We will construct $\mu^{l,k}$ in a lexicographic order with respect to indices $(l,k)$, for $l=c+r+1 \cdots N, k=0\cdots l$, together with positive integers $\delta_l$.

For simplicity, we start by setting:
\begin{equation}
\label{delta_(c+r+1)}
\delta_{c+r+1}\,
\geqslant\,
\max\,
\{
\epsilon_1,
\dots,
\epsilon_{c+r}
\}.
\end{equation} 
For every $l=c+r+1 \cdots N$, in this step, we begin with choosing $\mu_{l,0}$ that satisfies:
\begin{equation}
\label{mu_0>?}
\explain{see~\thetag{\ref{value of heart, nu, 2}},
\thetag{\ref{value of heart, nu}}}
\qquad
\mu_{l,0}
\geqslant
l\,
\delta_l
+
l\,\delta_{c+r+1}
+
l
+
1
+
l\,\varheartsuit,
\end{equation}
then  inductively we choose $\mu_{l,k}$ satisfying:
\begin{equation}
\label{mu_k>?}
\explain{see~\thetag{\ref{negative twisted degree, tau, rho, vanishing coordinates}},
\thetag{\ref{value of heart, tau-rho}}}
\qquad
\mu_{l,k}\,
\geqslant\,
\sum_{j=0}^{k-1}\,
l\,\mu_{l,j}
+
(l-k)\,\delta_l
+
l\,\delta_{c+r+1}
+
l
+
1
+
l\,\varheartsuit
\qquad
{\scriptstyle{(k\,=\,1\,\cdots\,l)}}.
\end{equation}
If $l<N$, we end this step by setting:
\begin{equation}
\label{delta_l=?}
\delta_{l+1}
:=
l\,\mu_{l,l}
\end{equation}
as the starting point for the next step $l+1$. At the end $l=N$, we require that:
\[
d\,
\geqslant\,
(N+1)\,\mu_{N,N}
\]
be large enough.

\subsection{Controlling the base loci}
\label{subsection: Controlling the base loci}
We will provide some technical preparations in Section~\ref{section: Some Technical Details}.

By
adapting the arguments in~\cite[Section 9]{Xie-2015-arxiv},
we can show that, for generic choices
of parameters $A_{\bullet}^{\bullet}, M_{\bullet}^{\bullet;\bullet}$,
firstly, the:
\begin{equation}
\label{base locus of global symmetric forms}
\text{Base Locus of }
\Big\{
\phi_{j_1,\dots,j_n}^{\nu},
\psi_{j_1,\dots,j_n}^{\tau,\,\rho}
\Big\}
^{
\nu,\,
\tau,\,
\rho
}_{
1
\leqslant 
j_1 
<
\cdots 
<
j_n 
\leqslant 
c
}\,
=:\,
\mathsf{BS}
\end{equation}
is discrete\big/empty over the `coordinates nonvanishing part'
$
\{
\zeta_0\cdots \zeta_N
\neq
0
\}
$,
and secondly, 
for every $1\leqslant \eta \leqslant n-1$, for every 
$
0\leqslant v_1<\dots<v_{\eta}\leqslant N
$, the:
\begin{equation}
\label{base loci over coordinates vanishing part}
\text{Base Locus of }
\Big\{
{}_{v_1,\dots,v_{\eta}}\phi_{j_1,\dots,j_{n-\eta}}^{\nu},\,
{}_{v_1,\dots,v_{\eta}}\psi_{j_1,\dots,j_{n-\eta}}^{\tau,\,\rho}
\Big\}
_{
1
\leqslant 
j_1 
<
\cdots 
<
j_{n-\eta} 
\leqslant 
c
}
^{
\nu,\,
\tau,\,
\rho
}\,
=:\,
{}_{v_1,\dots,v_{\eta}}\mathsf{BS}
\end{equation}
is discrete\big/empty over the corresponding `coordinates nonvanishing part' 
$
\{
\zeta_{r_0}\cdots \zeta_{r_{N-\eta}}
\neq
0
\}
$.

For the sake of completeness, we sketch the 
proof in Subsection~\ref{subsection: Emptiness of base loci} below.

\subsection{Effective degree estimates}
\label{subsection: Effective degree estimates}
In the Algorithm above, we first set $\varheartsuit=2$, $\epsilon_1=\cdots=\epsilon_{c+r}=1$,
and next we demand all inequalities~\thetag{\ref{delta_(c+r+1)}} --
\thetag{\ref{delta_l=?}}  to be exactly equalities.
Thus we receive the estimate
(cf.~\cite[Section 11]{Xie-2015-arxiv}):
\[
(N+1)\,\mu_{N,N}
<\,
N^{N^2/2}-1:=\texttt{d}_0\,
\qquad
{\scriptstyle(\forall\,N\,\geqslant\,3)}.
\] 

Now, recall the value $\epsilon_0=3/\texttt{d}_0$ in Definition~\ref{almost parallel}.
In fact, the motivation is the following

\begin{Proposition}
\label{explain the definition of proportional}
Let $\mathcal{L}$, $\mathcal{S}$ be two ample line bundles on $\mathbf{P}$.
Then $\mathcal{S}$ is almost proportional to $\mathcal{L}$
if and only if there exist some positive integers $d\geqslant \texttt{d}_0$, $s, l\geqslant 1$, such that
$\mathcal{S}^{s}=\mathcal{A}\otimes\mathcal{L}^{l}\otimes \mathcal{L}^{l\,d}$,
where the line bundle $\mathcal{A}$ satisfies that
$\mathcal{A}\otimes\mathcal{L}^l$ is very ample and that
$\mathcal{A}\otimes\mathcal{L}^{-2\,l}<0$ is negative.
\end{Proposition}

\begin{proof}
``$\Longleftarrow$'' We can take
$\alpha=s\cdot [\mathcal{S}]$ and $\beta=l\,d\cdot [\mathcal{L}]$, so that $\alpha-\beta=[\mathcal{A}\otimes\mathcal{L}^l]>0$,
and that $(1+\epsilon_0)\,\beta-\alpha\geqslant  
(1+3/d)\,\beta-\alpha=-\,[\mathcal{A}
\otimes\mathcal{L}^{-2\,l}]>0$.

``$\Longrightarrow$'' Since $\mathbb{Q}_+$ is dense in
$\mathbb{R}_+$, we may assume that $\alpha\in \mathbb{Q}_{+}\cdot [\mathcal{S}]$
and $\beta\in \mathbb{Q}_{+}\cdot [\mathcal{L}]$.
Next, we can choose a sufficiently divisible integer $m>0$ such that
$m\cdot\alpha=s_0\cdot [\mathcal{S}]$ and $m\cdot\beta=l_0\,\texttt{d}_0\cdot [\mathcal{L}]$ for some positive integers $s_0, l_0>0$.
Set the line bundle $\mathcal{A}_0:=\mathcal{S}^{s_0}\otimes
(\mathcal{L}^{l_0}\otimes \mathcal{L}^{l_0\,\texttt{d}_0})^{-1}$, hence
$\mathcal{S}^{s_0}=\mathcal{A}_0\otimes\mathcal{L}^{l_0}\otimes \mathcal{L}^{l_0\,\texttt{d}_0}$.
Now, using
$\beta<\alpha<(1+\epsilon)\,\beta$, we receive:
{\footnotesize
\[
0
<
m
\cdot
(\alpha
-
\beta)
=
m\cdot\alpha-
m\cdot\beta
=
s_0\cdot [\mathcal{S}]
-
l_0\,\texttt{d}_0\cdot [\mathcal{L}]
=
[\mathcal{S}^{s_0}\otimes \mathcal{L}^{-l_0\,\texttt{d}_0}]
=
[\mathcal{A}_0\otimes\mathcal{L}^{l_0}],
\]
\[
0
>
m
\cdot 
\big[
\alpha
-
(1+\epsilon_0)\,\beta
\big]
=
m
\cdot 
\alpha
-
(1+3/\texttt{d}_0)\,
m
\cdot
\beta
=
s_0\cdot [\mathcal{S}]
-
(1+3/\texttt{d}_0)\,
l_0\,\texttt{d}_0\cdot [\mathcal{L}]
=
[\mathcal{A}_0\otimes\mathcal{L}^{-2\,l_0}].
\]
}\!\!
The first line above implies that
$(\mathcal{A}_0\otimes\mathcal{L}^{l_0})^{\otimes\,m'}$
is very ample for some positive integer $m'>0$.
Thus we can set $s:=s_0\,m'$, $l:=l_0\,m'$, $\mathcal{A}:=\mathcal{A}_0^{m'}$, then 
$\mathcal{S}^{s}=\mathcal{A}\otimes\mathcal{L}^{l}\otimes \mathcal{L}^{l\,\texttt{d}_0}$ satisfies that
$\mathcal{A}\otimes\mathcal{L}^l$ is very ample and that
$\mathcal{A}\otimes\mathcal{L}^{-2\,l}<0$ is negative.
\end{proof}

\begin{Remark}
\label{we can take d=d_0}
In the above proof, we see that the second assertion holds for
$d=\texttt{d}_0$.
In fact, it holds for any positive integer $d'\leqslant d$, since we have:
\[
\mathcal{L}^{s\,(1+d')}\,
=\,
\big(
\mathcal{A}
\otimes 
\mathcal{L}^{l\,(1+d)}
\big)^{1+d'}\,
=\,
\mathcal{A}^{1+d'}
\otimes
\mathcal{L}^{l\,(1+d)}
\otimes 
\big(
\mathcal{L}^{l\,(1+d)}
\big)^{d'},
\]
where: 
\[
\mathcal{A}^{1+d'}
\otimes
\mathcal{L}^{l\,(1+d)}\,
=\,
\big(
\underbrace{
\mathcal{A}
\otimes
\mathcal{L}^{l}
}_{
\text{very ample}
}
\big)^{1+d'}
\otimes
\mathcal{L}^{l\,(d-d')}
\] 
is very ample and where:
\[
\mathcal{A}^{1+d'}
\otimes
\mathcal{L}^{-2\,l\,(1+d)}\,
=\,
\big(
\underbrace{
\mathcal{A}\otimes\mathcal{L}^{-2\,l}
}_{
\text{negative}
}
\big)^{1+d'}
\otimes
\mathcal{L}^{-2\,l\,(d-d')}\
<\
0.
\]
Thus the second assertion holds not only for
$(d,s,l)$ but also for
$\big(d',s\,(1+d'),l\,(1+d)\big)$.
\end{Remark}

\subsection{Proof of Theorem~\ref{Main Theorem, 2nd}}
\label{subsection: proof of the main theorem}
Summarizing the above Subsections~\ref{subsection: Adaptation} -- \ref{subsection: Effective degree estimates}, we can obtain 

\begin{Theorem}
\label{general theorem of almost the same degree case}
Let $\mathbf{P}$ be a smooth projective
variety of dimension $N$,
and let
$\mathcal{L}$ be a very ample line bundle
over $\mathbf{P}$. For any integers
$c, r\geqslant 0$ with $2c+r\geqslant N$,
for any integer $d\geqslant \texttt{d}_0$, 
for any $c+r$ line bundles $\mathcal{A}_i$ $(i=1\cdots c+r)$ such
that
$
\mathcal{A}_i
\otimes
\mathcal{L}$
are very ample and that
$
\mathcal{A}_i
\otimes
\mathcal{L}^{-\,2}
<
0
$,
setting:
\[
\mathcal{L}_i\,
=\,
\mathcal{A}_i
\otimes 
\mathcal{L}
\otimes
\mathcal{L}^{d}
\qquad
{\scriptstyle (i\,=\,1\,\cdots\, c+r)},
\]
then, 
for generic $c+r$
 hypersurfaces: 
\[
H_1
\in
\big|\mathcal{L}_1\big|,
\dots,
H_{c+r}\in \big|\mathcal{L}_{c+r}\big|,
\]
the cotangent bundle 
$\Omega_V$ of the intersection of the first $c$ hypersurfaces
$
V
=
H_1
\cap
\cdots
\cap 
H_c
$
restricted 
to the intersection of all the $c+r$ hypersurfaces
$
X
=
H_1
\cap
\cdots
\cap 
H_c
\cap
H_{c+1}
\cap
\cdots
\cap 
H_{c+r}
$
is ample.
\end{Theorem}

Denote 
the projectivization of the cotangent bundle $\Omega_{\mathbf{P}}$ of $\mathbf{P}$
by: 
\[
\mathbb{P}(\Omega_{\mathbf{P}})\,
:=\,
\mathsf{Proj}\,
\big(
\oplus_{k\geqslant 0}\mathsf{Sym}^k\,
\Omega_{\mathbf{P}}
\big),
\]
and denote 
the associated Serre line bundle by $\mathcal{O}_{\mathbb{P}(\Omega_{\mathbf{P}})}(1)$.
For any integers $a, b\geqslant 0$, for any $a+b$ global sections $F_1$, $\dots$, $F_{a}$, $F_{a+1}$, $\dots$, $F_{a+b}$ of arbitrary $a+b$ line bundles over $\mathbf{P}$, denote by:
\[
{}_{F_{a+1},\dots,F_{a+b}}\mathbb{P}_{F_1,\dots,F_{a}}\
\subset\
\mathbb{P}(\Omega_{\mathbf{P}}).
\]
the unique subscheme defined by
 equations $F_1, \dots, F_{a+b}, \dformal F_1,\dots, \dformal F_{a}$.
Thus, we reformulate the above theorem as:

\medskip
\noindent
{\bf Theorem~\ref{general theorem of almost the same degree case}'.}
{
\em
For generic $c+r$ sections: 
\[
F_1
\in
\mathsf{H}^0\,(\mathbf{P}, \mathcal{L}_1),
\dots,
F_{c+r}
\in
\mathsf{H}^0\,(\mathbf{P}, \mathcal{L}_{c+r}),
\]
the Serre line bundle
$\mathcal{O}_{\mathbb{P}(\Omega_{\mathbf{P}})}(1)$
is ample over the subvariety
${}_{F_{c+1},\dots,F_{c+r}}\mathbb{P}_{F_1,\dots,F_c}$.
}

\begin{proof}[Proof of Theorem~\ref{general theorem of almost the same degree case}]
We may assume that $N\geqslant 3$ and $c+r<N$, otherwise there is nothing to prove. Set $n=N-c-r$, observe that $1\leqslant n\leqslant c$.
Since ampleness is a Zariski open condition in family (Grothendieck), 
we only need to provide one ample example $H_1,\dots,H_{c+r}$. In fact, we will construct $c+r$ sections $F_1, \dots, F_{c+r}$  of the  MCM shape~\thetag{\ref{F_i-moving-coefficient-method-full-strenghth}} to conclude the proof.

\smallskip
\noindent
{\em Step 1.}
Since $d\geqslant \texttt{d}_0$, by the effective degree estimates in preceding subsection, 
we can construct integers $\{\mu^{l,k}\}$ that satisfy
the Algorithm in Subsection~\ref{Subsection: Delayed Algorithm}. Now, the structure of~\thetag{\ref{F_i-moving-coefficient-method-full-strenghth}} is fixed,
and we will choose some appropriate coefficients $A_i^{\bullet}, M_i^{\bullet;\bullet}$ for $i=1\cdots c+r$.

\smallskip
\noindent
{\em Step 2.}
For generic choices
of parameters $A_{\bullet}^{\bullet}, M_{\bullet}^{\bullet;\bullet}$, both $X, V$ are smooth complete, and moreover, 
for all $1\leqslant \eta \leqslant n=N-c-r$, for all indices 
$
0\leqslant v_1<\dots<v_{\eta}\leqslant N
$, the further intersection varieties
${}_{v_1,\dots,v_{\eta}}X$ are all smooth complete.
The reasoning is much the same as in Bertini's Theorem.
For the sake of completeness, we provide a proof in
Subsection~\ref{subsection: Bertini-type assertions} below.

\smallskip
\noindent
{\em Step 3.}
For generic choices
of parameters $A_{\bullet}^{\bullet}, M_{\bullet}^{\bullet;\bullet}$,
all the constructed negatively twisted symmetric differential forms have discrete based loci outside `coordinates vanishing part',
see Subsection~\ref{subsection: Controlling the base loci} for details.
This is the core of the moving coefficients method.

\smallskip
\noindent
{\em Step 4.}
Choose any generic parameters $A_{\bullet}^{\bullet}, M_{\bullet}^{\bullet;\bullet}$
that satisfy the properties in the above two steps.
We claim that the corresponding sections $F_1,\dots, F_{c+r}$
constitute one ample example.

\begin{proof}[Proof of the claim]
Abbreviate $\mathbb{P}:={}_{F_{c+1},\dots,F_{c+r}}\mathbb{P}_{F_1,\dots,F_c}$ and  ${}_{v_1,\dots,v_{\eta}}\mathbb{P}:={}_{F_{c+1},\dots,F_{c+r},\zeta_{v_1},\dots,\zeta_{v_{\eta}}}\mathbb{P}_{F_1,\dots,F_c}$. Let
$
\pi
\colon
\mathbb{P}(\Omega_{\mathbf{P}})
\longrightarrow 
\mathbf{P}
$
be the canonical projection.
Note that all the obtained symmetric differential forms in Step 3 can be viewed as sections (when $\eta=0$, we agree ${}_{v_1,\dots,v_{\eta}}\mathbb{P}=\mathbb{P}$):
\begin{equation}
\label{all great sections}
{}_{v_1,\dots,v_{\eta}}\omega_{\text{\ding{100}}}\
\in\
\mathsf{H}^0
\big(
{}_{v_1,\dots,v_{\eta}}\mathbb{P}, \mathcal{O}_{\mathbb{P}(\Omega_{\mathbf{P}})}(n-\eta)\otimes \pi^*{}_{v_1,\dots,v_{\eta}}\mathcal{L}_{\text{\ding{100}}}
\big),
\end{equation}
where we always use \text{\ding{100}} to denote auxiliary integers, and where all
${}_{v_1,\dots,v_{\eta}}\mathcal{L}_{\text{\ding{100}}}<0$ are some negative line bundles.
Choose an ample $\mathbb{Q}$-divisor $\mathcal{S}>0$
over $\mathbf{P}$ such that all ${}_{v_1,\dots,v_{\eta}}\mathcal{L}_{\text{\ding{100}}}/(n-\eta)+
\mathcal{S}<0$ are still negative. Then we claim that
$\mathcal{N}:=\mathcal{O}_{\mathbb{P}(\Omega_{\mathbf{P}})}(1)\otimes \pi^*\mathcal{S}^{-1}$ is nef over
$\mathbb{P}$. 

Indeed, for any irreducible curve
$C\subset \mathbb{P}$, if $C$ lies in at least
$n$ `coordinate hyperplanes' defined by $\zeta_{v_1}, \dots, \zeta_{v_n}$, then by Step 2 we see that $C$ must contract to a point by $\pi$, thus
$\mathcal{N}\big{\vert}_C\cong\mathcal{O}_{\mathbb{P}(\Omega_{\mathbf{P}})}(1)\big{\vert}_C$
 is not only nef but ample.
Assume now that 
$C$ lies in at best
$\eta<n$ `coordinate hyperplanes' defined by $\zeta_{v_1}, \dots, \zeta_{v_\eta}$ ($\eta$ could be zero).
Since the base locus of all sections in~\thetag{\ref{all great sections}}
is discrete over the `coordinates nonvanishing part'
$
\{
\zeta_{r_0}\cdots \zeta_{r_{N-\eta}}
\neq
0
\}
$, 
and $C\cap \{
\zeta_{r_0}\cdots \zeta_{r_{N-\eta}}
\neq
0
\}$ is one-dimensional,
we can find some
${}_{v_1,\dots,v_{\eta}}\omega_{\text{\ding{100}}}$ such that ${}_{v_1,\dots,v_{\eta}}\omega_{\text{\ding{100}}}\big{\vert}_C\neq 0$. Thus the intersection number  $C\,\cdot\,\mathcal{O}_{\mathbb{P}(\Omega_{\mathbf{P}})}(n-\eta)\otimes \pi^*{}_{v_1,\dots,v_{\eta}}\mathcal{L}_{\text{\ding{100}}}$ 
is $\geqslant 0$. Since 
${}_{v_1,\dots,v_{\eta}}\mathcal{L}_{\text{\ding{100}}}/(n-\eta)+
\mathcal{S}<0$, we immediately conclude that
$C\cdot\mathcal{N}\geqslant 0$. 

Lastly, since $\mathcal{S}>0$ over $\mathbf{P}$,
there exists some large integer $m\gg 1$ such that 
$\mathcal{P}:=\mathcal{O}_{\mathbb{P}(\Omega_{\mathbf{P}})}(1)\otimes \pi^*\mathcal{S}^{m}>0$ is positive
over $\mathbb{P}(\Omega_{\mathbf{P}})$. In particular,
it is also positive over $\mathbb{P}$.
Since `nef+ample=ample', we have 
$m\, \mathcal{N}+\mathcal{P}> 0$ over $\mathbb{P}$,
that is $\mathcal{O}_{\mathbb{P}(\Omega_{\mathbf{P}})}(1)\big{\vert}_{\mathbb{P}}>0$.
\end{proof}
Thus we conclude the proof.
\end{proof}

Finally, using the product coup, we obtain 

\begin{proof}
[Proof of Theorem~\ref{Main Theorem, 2nd}]
For every $i=1\cdots c+r$,
since $\mathcal{L}_i$ is almost
proportional to $\mathcal{L}$, by Proposition~\ref{explain the definition of proportional},
there exist some positive integers $s_i, l_i\geqslant 1$, $d_i\geqslant \texttt{d}_0$ such that
$\mathcal{L}_i^{s_i}
=
\mathcal{A}_i
\otimes
\mathcal{L}^{l_i}
\otimes 
\mathcal{L}^{l_i\,d_i}$,
where the line bundle $\mathcal{A}_i$ satisfies that
$\mathcal{A}_i\otimes\mathcal{L}^{l_i}$ is very ample and that
$\mathcal{A}_i\otimes\mathcal{L}^{-2\,l_i}<0$ is negative.
In order to apply Theorem~\ref{general theorem of almost the same degree case}',
first of all, we need an

\begin{Observation}
\label{technical observation}
There exist some positive integers $\widetilde{s}_1,\dots, 
\widetilde{s}_{c+r}, \ell\geqslant 1$ and $d\geqslant  \texttt{d}_0$ such that:
\[
\mathcal{L}_i^{\widetilde{s}_i}\,
=\,
\widetilde{\mathcal{A}_i}
\otimes 
\widetilde{\mathcal{L}}
\otimes
\widetilde{\mathcal{L}}^{d},
\qquad
\mathcal{L}_i^{\widetilde{s}_i+1}\,
=\,
\widetilde{\mathcal{B}_i}
\otimes 
\widetilde{\mathcal{L}}
\otimes
\widetilde{\mathcal{L}}^{d}
\qquad
{\scriptstyle (i\,=\,1\,\cdots\, c+r)},
\]
where $\widetilde{\mathcal{L}}:=\mathcal{L}^{\ell}$
is very ample,
and where
$
\widetilde{\mathcal{A}_i}
\otimes 
\widetilde{\mathcal{L}}$,
$\widetilde{\mathcal{B}_i}
\otimes 
\widetilde{\mathcal{L}}$
are very ample, and where
$
\widetilde{\mathcal{A}_i}
\otimes 
\widetilde{\mathcal{L}}^{-2}
$,
$
\widetilde{\mathcal{B}_i}
\otimes 
\widetilde{\mathcal{L}}^{-2}
$
are negative.
\end{Observation}

\begin{proof}
First, by Remark~\ref{we can take d=d_0},
we may assume that
$d_1=\cdots=d_{c+r}=d\geqslant\texttt{d}_0$.

Next, we may assume that $l_1=\cdots=l_{c+r}=l$.
Otherwise, we can choose a positive integer
$l$ which is divisible by $l_1,\dots,l_{c+r}$,
then we receive\big/rewrite:
\[
\mathcal{L}_i^{s_i\,l/l_i}
=
\big(
\mathcal{A}_i
\otimes
\mathcal{L}^{l_i}
\otimes 
\mathcal{L}^{l_i\,d}
\big)^{l/l_i}
=
\mathcal{A}_i^{l/l_i}
\otimes
\mathcal{L}^{l}
\otimes 
\mathcal{L}^{l\,d}
\qquad
{\scriptstyle (i\,=\,1\,\cdots\, c+r)},
\]
while 
$\mathcal{A}_i^{l/l_i}
\otimes
\mathcal{L}^{l}
=
\big(
\mathcal{A}_i
\otimes
\mathcal{L}^{l_i}
\big)^{l/l_i}
$
remains very ample and also
$\mathcal{A}_i^{l/l_i}
\otimes
\mathcal{L}^{-2\,l}
=
\big(
\mathcal{A}_i
\otimes
\mathcal{L}^{-2\,l_i}
\big)^{l/l_i}
<
0
$.

Lastly, we can choose one large integer $m\gg 1$ such that,
for all $i=1\cdots c+r$,
not only
$
\mathcal{L}_i
\otimes
(
\mathcal{A}_i
\otimes
\mathcal{L}^{l}
)^{m}
$
are very ample,
but also
$
\mathcal{L}_i
\otimes
(
\mathcal{A}_i
\otimes
\mathcal{L}^{-2\,l}
)^{m}
<
0
$
are negative.
Thus, the following data:
\[
\ell\,
:=\,
m\,l,
\quad
\widetilde{s}_i\,
:=\,
m\,s_i,
\quad
\widetilde{\mathcal{A}_i}\,
:=\,
\mathcal{A}_i^{m},
\quad
\widetilde{\mathcal{B}_i}\,
:=\,
\mathcal{L}_i
\otimes
\mathcal{A}_i^{m}
\qquad
{\scriptstyle (i\,=\,1\,\cdots\, c+r)}
\]
satisfy the claimed observation.
\end{proof}

Now, we can set:
\begin{equation}
\label{product formula d=?}
\texttt{d}\,
=\,
\texttt{d}(\mathcal{L}_{1},\dots,
\mathcal{L}_{c+r}, \mathcal{L})\,
=\,
\max_{1\leqslant i\leqslant c+r}\,
\big\{
\widetilde{s}_i\,(\widetilde{s}_i-1)
\big\}.
\end{equation}
For any integers 
$
d_1,
\dots,
d_{c+r}\, 
\geqslant\, 
\texttt{d}
$,
all of them can be written as:
\[
d_i\,
=\,
p_i\,\widetilde{s}_i
+
q_i\,(\widetilde{s}_i+1)
\qquad
{\scriptstyle (i\,=\,1\,\cdots\, c+r)}
\]
for some integers $p_i,q_i\geqslant 0$.
Let every:
\[
F_i
:=
f^i_{1}
\cdots
f^i_{p_i}\,
f^i_{p_i+1}
\cdots
f^{i}_{p_i+q_i}\
\in\
\mathsf{H}^0\,
\big(
\mathbf{P},\,
\mathcal{L}_i^{d_i}
\big)
\]
be a product of some sections:
\[
f^i_{1},
\dots,
f^i_{p_i}\
\in\
\mathsf{H}^0\,
\big(
\mathbf{P},\,
\mathcal{L}_i^{\widetilde{s}_i}
\big),
\quad
f^i_{p_i+1},
\dots,
f^i_{p_i+q_i}\
\in\
\mathsf{H}^0\,
\big(
\mathbf{P},\,
\mathcal{L}_i^{\widetilde{s}_i+1}
\big)
\]
to be chosen, then the product coup reveals the decomposition:
\[
\aligned
{}_{F_{c+1},\dots,F_{c+r}}\mathbb{P}_{F_1,\dots,F_c}\,
=\,
&
\cup_{k=0\cdots c}
\cup_{1\leqslant i_1<\cdots < i_k \leqslant c}
\cup_{
\substack{
1\leqslant v_{i_j}\leqslant p_{v_j}+q_{v_j}
\\
j=1\cdots k
}
}
\cup_{
\substack{
\{
r_1,\dots,r_{c-k}
\}
=
\{1,\dots,c\}
\setminus
\{
i_1,\dots,i_k
\}
\\
1
\leqslant 
w_{r_l}^1
<
w_{r_l}^2
\leqslant
p_{r_l}
+
q_{r_l}
\\
l=1\cdots c-k
}
}
\cup_{
\substack{
1
\leqslant
u_j
\leqslant
p_j+q_j
\\
j=c+1\cdots c+r
}
}
\\
&
\ \ \ \ \ \ \ \ \ \ \ \ \ \
{}_{f_{w_{r_1}^1}^{r_1},f_{w_{r_1}^2}^{r_1},\dots,
f_{w_{r_{c-k}}^1}^{r_{c-k}},f_{w_{r_{c-k}}^2}^{r_{c-k}},
f_{u_{c+1}}^{c+1},\dots,f_{u_{c+r}}^{c+r}
}
\mathbb{P}_{f_{v_{i_1}}^{i_1},\dots,f_{v_{i_k}}^{i_k}}.
\endaligned
\]
Now, applying Theorem~\ref{general theorem of almost the same degree case}', for generic choices of
$\{f_{\bullet}^{\bullet}\}$, the Serre line bundle
$\mathcal{O}_{\mathbb{P}(\Omega_{\mathbf{P}})}(1)$
is ample on every subscheme
$
{}_{f_{w_{r_1}^1}^{r_1},f_{w_{r_1}^2}^{r_1},\dots,
f_{w_{r_{c-k}}^1}^{r_{c-k}},f_{w_{r_{c-k}}^2}^{r_{c-k}},
f_{u_{c+1}}^{c+1},\dots,f_{u_{c+r}}^{c+r}
}
\mathbb{P}_{f_{v_{i_1}}^{i_1},\dots,f_{v_{i_k}}^{i_k}}
$, and therefore
is also ample on their union 
${}_{F_{c+1},\dots,F_{c+r}}\mathbb{P}_{F_1,\dots,F_c}$.
Since ampleness is a generic property in family, we 
conclude the proof.
\end{proof}

\subsection{Effective lower degree bound $N^{N^2}$ of Theorem~\ref{Gentle generalization of Debarre conjecture}}
\label{subsection: proof of theorem 2}
Now, we provide an effective degree estimate of
Theorem~\ref{Main Theorem, 2nd} in the case $\mathcal{L}_1=\cdots=\mathcal{L}_{c+r}=\mathcal{L}$.

When $N=1, 2$, Theorem~\ref{Main Theorem, 2nd} holds trivially
for $\texttt{d}=N^{N^2}$.
When $N\geqslant 3$,
denote the trivial line bundle on $\mathbf{P}$ by
$\mathbf{0}_{\mathbf{P}}$.
Note that in Observation~\ref{technical observation}
we can take
$\widetilde{s}_1=\cdots= 
\widetilde{s}_{c+r}=\texttt{d}_0+1$, $\ell=1$, so that:
\[
\mathcal{L}^{\texttt{d}_0+1}\,
=\,
\mathbf{0}_{\mathbf{P}}
\otimes
\mathcal{L}
\otimes
\mathcal{L}^{\texttt{d}_0},
\qquad
\mathcal{L}^{\texttt{d}_0+2}\,
=\,
\mathcal{L}
\otimes
\mathcal{L}
\otimes
\mathcal{L}^{\texttt{d}_0}
\]
satisfy the requirements.
Thus by~\thetag{\ref{product formula d=?}}
we can set:
\[
\texttt{d}\,
=\,
\texttt{d}(\mathcal{L})\,
=\,
\max_{1\leqslant i\leqslant c+r}\,
\big\{
\widetilde{s}_i\,(\widetilde{s}_i-1)
\big\}\,
=\,
\texttt{d}_0\,(\texttt{d}_0+1)
=
(N^{N^2/2}-1)\,
N^{N^2/2}
<
N^{N^2}.
\]
In particular, when $r=0$, we recover
Theorem~\ref{Gentle generalization of Debarre conjecture}.

\section{\bf Some Technical Details}
\label{section: Some Technical Details}
\subsection{Surjectivity of evaluation maps}
Recalling 
the notation in
Definition~\ref{define formal differential},
at every closed point $z\in \mathbf{P}$, for every tangent vector $\xi\in\mathrm{T}_{\mathbf{P}}\big{\vert}_z$,
we can choose any local trivialization $(U,s)$ 
of the line bundle $\mathcal{S}$
near point $z$, and then evaluate $S$, $\dformal S$
at $(z,\xi)$ by:
\[
\aligned
S(z)\,(U,s)\,
&
:=\,
S/s\,(z)\
\in\
\mathbb{K}, 
\\
\dformal S(z,\xi)\,(U,s)\,
&
:=\,
\mathrm{d}\,
(S/s)\,
(z,\xi)\
\in\
\mathbb{K}.
\endaligned
\] 
If $(U,s')$ is another local trivialization of $\mathcal{S}$,
then we have the transition formula:
\begin{equation}
\label{surjectivity is independent of trivializations}
\begin{pmatrix}
S
\\
\dformal S
\end{pmatrix}\,
(z,\xi)\,
(U,s)\,
=\,
\underbrace{
\begin{pmatrix}
s'/s
&
0
\\
\mathrm{d}\,(s'/s)
&
s'/s
\end{pmatrix}
}_{
\text{invertible}
}\,
\cdot\,
\begin{pmatrix}
S
\\
\dformal
S
\end{pmatrix}\,
(z,\xi)\,
(U,s')
\end{equation}
Thanks to the above identity, 
in assertions which do not depend on the particular choice of $(U, s)$,
we can just write $S(z)$, $\dformal S(z,\xi)$
by dropping $(U,s)$.
 
\begin{Proposition}
Let $\mathcal{S}$ be a very-ample line bundle over a smooth $\mathbb{K}$-variety $\mathbf{P}$. Then one has:

\begin{itemize}
\item[\bf (i)]
at every closed point $z\in \mathbf{P}$, for any nonzero tangent vector $0\neq\xi\in\mathrm{T}_{\mathbf{P}}\big{\vert}_z$,
the evaluation map:
\[
\aligned
\begin{pmatrix}
v_z
\\
d_z(\xi)
\end{pmatrix}\,
\colon\
\mathsf{H}^0
(
\mathbf{P},
\mathcal{S}
)\,
&
\longrightarrow\,
\mathbb{K}^2
\\
S\,
&
\longmapsto\,
\big(
S(z), 
\dformal
S(z,\xi)
\big)^{\mathrm{T}}
\endaligned
\]
is surjective;

\smallskip
\item[\bf (ii)]
at every closed point $z\in \mathbf{P}$, for any 
$N=\dim\,\mathbf{P}$ linearly independent tangent vectors $\xi_1,\dots,\xi_N\in\mathrm{T}_{\mathbf{P}}\big{\vert}_z$,
the evaluation map:
\[
\aligned
\begin{pmatrix}
v_z
\\
d_z(\xi_1)
\\
\vdots
\\
d_z(\xi_N)
\end{pmatrix}\,
\colon\
\mathsf{H}^0
(
\mathbf{P},
\mathcal{S}
)\,
&
\longrightarrow\,
\mathbb{K}^{N+1}
\\
S\,
&
\longmapsto\,
\big(
S(z), 
\dformal
S(z,\xi_1),
\dots,
\dformal
S(z,\xi_N)
\big)^{\mathrm{T}}
\endaligned
\]
is surjective.
\end{itemize}
\end{Proposition}

\begin{proof}
We have the following three elementary observations.

\begin{itemize}
\item[(1)]
By transition formula~\thetag{\ref{surjectivity is independent of trivializations}},
property {\bf (i)} is independent of the choice of local trivialization $(U,s)$ of $\mathcal{S}$ near $z$,
so it makes sense.

\smallskip
\item[(2)]
In any fixed local trivialization $(U,s)$ of $\mathcal{S}$ near $z$, by basic linear algebra,
properties {\bf (i)}, {\bf (ii)} are equivalent to each other.

\smallskip
\item[(3)]
Property {\bf (i)} is the usual property
of `very-ampleness'.
\end{itemize}

Thus we may conclude the proof
by the reasoning `\text{very-ampleness}' $\Longrightarrow$ {\bf (i)}
$\Longleftrightarrow$ {\bf (ii)}.
\end{proof}

\begin{Proposition}
\label{surjectivity lemma}
Let $\mathcal{S}$ be a very ample line bundle over a smooth $\mathbb{K}$-variety $\mathbf{P}$, and let 
$\mathcal{A}$ be any line bundle over $\mathbf{P}$
with a nonzero section
$A\neq 0$. Then, at every closed point $z\in \mathsf{D}(A)\subset\mathbf{P}$, for any nonzero tangent vector $0\neq\xi\in\mathrm{T}_{\mathbf{P}}\big{\vert}_z$,
the evaluation map:
\[
\aligned
\begin{pmatrix}
A\cdot v_z
\\
d_z(A\cdot\,)(\xi)
\end{pmatrix}\,
\colon\
\mathsf{H}^0
(
\mathbf{P},
\mathcal{S}
)\,
&
\longrightarrow\,
\mathbb{K}^2
\\
S\,
&
\longmapsto\,
\Big(
(A\cdot S)\,(z), 
\dformal
(A\cdot S)\,(z,\xi)
\Big)^{\mathrm{T}}
\endaligned
\]
is surjective.
\end{Proposition}

\begin{proof}
It is a direct consequence of the formula:
\[
\begin{pmatrix}
A\cdot v_z
\\
d_z(A\cdot\,)(\xi)
\end{pmatrix}\,
=\,
\underbrace{
\begin{pmatrix}
A(z)
&
0
\\
\dformal
A
(z,\xi)
&
A(z)
\end{pmatrix}
}_{
\text{invertible since $A(z)\neq 0$}
}\,
\cdot\,
\underbrace{
\begin{pmatrix}
 v_z
\\
d_z(\xi)
\end{pmatrix}
}_{
\text{surjective}
}
\qquad
\explain{Leibniz's rule}
\]
and of the preceding proposition.
\end{proof}

\subsection{Bertini-type assertions}
\label{subsection: Bertini-type assertions}
Recalling~\thetag{\ref{F_i-moving-coefficient-method-full-strenghth}} and that for all $i=1\cdots c+r$ the line bundles
$\mathcal{A}_i\otimes \mathcal{L}^{\epsilon_i}$ are very ample, we now fulfill the
{\em step 2} in the proof of Theorem~\ref{general theorem of almost the same degree case}. 
We start with

\smallskip
\noindent
{\sl Observation 1.}
`Smooth complete' is a Zariski open condition in family.

\smallskip
\noindent
{\sl Observation 2.}
We only need to prove
that, for generic choices of
$A_{1}^{\bullet}, M_{1}^{\bullet;\bullet}$,
the hypersurface
$H_1=\{F_1=0\}\subset \mathbf{P}$ is smooth complete.

\begin{proof}
Indeed, replacing $\mathbf{P}$ by $H_1$,
we can repeat the same argument to choose $A_{2}^{\bullet}, M_{2}^{\bullet;\bullet}$,
and so on.
Thus we know that there exists at least one choice of
parameters 
$A_{\bullet}^{\bullet}, M_{\bullet}^{\bullet;\bullet}$
such that $X, V$ are both smooth complete.
Immediately, by Observation 1 above, it holds for generic choices of parameters.

Next, to show that generically 
${}_{v_1,\dots,v_{\eta}}X$ is smooth complete,
we can start with 
${}_{v_1,\dots,v_{\eta}}\mathbf{P}$ instead of 
$\mathbf{P}$,
and use the same reasoning to conclude the proof.
\end{proof}

\noindent
{\sl Observation 3.}
We can first set all $M_{1}^{\bullet;\bullet}=0$, and
then thanks to
the following proposition, we can
 find some appropriate $A_{1}^{\bullet}$ 
such that $H_1$ is smooth complete.
Thus we finish the proof of {\em step 2}.

\begin{Proposition}
\label{Bertini-type assertion}
Let $\mathbf{P}$ be a smooth $\mathbb{K}$-variety 
of dimension $N$, and 
let $\mathcal{A}$, $\mathcal{B}$ be two line bundles over $\mathbf{P}$.
Assume that $\mathcal{A}$ is very ample, and that
$\mathcal{B}$ has $N+1$ global sections $B_0,\dots,B_N$
having empty common base locus.
Then, for generic choices
of parameters $A_0,\dots,A_N\in\mathsf{H}^0(\mathbf{P},\mathcal{A})$,
the section:
\[
F\,
=\,
\sum_{j=0}^N\,
A_j\,
B_j\
\in\
\mathsf{H}^0\,(\mathbf{P},\,\mathcal{A}\otimes \mathcal{B})
\]
defines a smooth complete subvariety.
\end{Proposition}

\smallskip
\noindent
{\em Proof.} 
Denoting 
$
\mathcal{M}
:=
\oplus_{j=0}^N\,
\mathsf{H}^0(\mathbf{P},\mathcal{A})
$, then $\mathbb{P}(\mathcal{M})$
stands for the projective parameter space
of $t=(A_0, \dots, A_N)$.
Now we introduce the universal subvariety:
\[
\mathsf{S}\,
:=\,
\Big\{
([t],z)\,
\colon\,
F_t(z)=0,\,
\dformal
F_t\,(z,\xi)=0,\,
\forall\,
\xi\in
\mathrm{T}_{\mathbf{P}}\big{|}_z
\Big\}
\
\subset\
\mathbb{P}(\mathcal{M})
\times
\mathbf{P}
\]
consisting of singular points.
We claim that $\dim\,\mathsf{S}< \dim\,\mathbb{P}(\mathcal{M})$. It suffices to show that, for every closed point $z\in\mathbf{P}$,
the fibre $\mathsf{S}_z\subset \mathbb{P}(\mathcal{M})
\times \{z\}\cong \mathbb{P}(\mathcal{M})$ over
$z$ satisfies that $\cdim\,\mathsf{S}_z>\dim\,\mathbf{P}$.

Indeed, choose $N$ linearly independent tangent vector
$\xi_1,\dots, \xi_N$ at point $z$,
and then consider the formal $\mathbb{K}$-linear map:
\[
\aligned
\mathsf{ev}\,
\colon\
\mathcal{M}\,
&
\longrightarrow\,
\mathbb{K}^{N+1}
\\
t\,
&
\longmapsto\,
\big(
F_t(z), 
\dformal
F_t\,(z,\xi_1),
\dots,
\dformal
F_t\,(z,\xi_N)
\big)^{\mathrm{T}}.
\endaligned
\]
By Proposition~\ref{surjectivity lemma},
$\mathsf{ev}$ is surjective. Note that
$\mathsf{S}_z\subset \mathbb{P}(\mathcal{M})$ consists of points $[t]\in\mathbb{P}(\mathcal{M})$ such that
$\mathsf{ev}(t)={\bf 0}$. Thus we see that:
\[
\cdim\,\mathsf{S}_z\,
=\,
\underline{
\cdim\,\big(\{{\bf 0}\}\subset \mathbb{K}^{N+1}\big)\,
=\,
N+1
}\,
>\,
N\,
=\,
\dim\,\mathbf{P}.
\eqno
\qed
\]

\medskip
We will see in the proof of Proposition~\ref{why the core lemma is essential} that 
the {\sl Core Lemma} of MCM plays the same role as that of
the above underlined codimension equality\big/estimate.

\subsection{Emptiness of the base loci}
\label{subsection: Emptiness of base loci}
Recalling~\thetag{\ref{base locus of global symmetric forms}},
\thetag{\ref{base loci over coordinates vanishing part}},
in order to characterize the base loci
\[
\mathsf{BS},\,\,
{}_{v_1,\dots,v_{\eta}}\mathsf{BS}\
\subset\
\mathbb{P}(\Omega_{\mathbf{P}}),
\]
we introduce the following subvarieties (cf.~\cite[p.~62, (148)]{Xie-2015-arxiv}):
\[
\mathscr{M}_{b}^a\ \
\subset\ \
{\sf Mat}_{b\times 2(a+1)}(\mathbb K)
\qquad
{\scriptstyle (\forall\,2\,\leqslant\,a\,\leqslant\, b)}
\]
consisting of all $b\times 2(a+1)$ matrices
$(\alpha_0\mid\alpha_1\mid\dots\mid\alpha_{a}\mid\beta_0\mid\beta_1\mid
\dots\mid\beta_{a})$ such that:
\begin{itemize}

\smallskip\item[{\bf (i)}]
the sum of all $2a+2$ colums is zero:
\begin{equation}
\label{sum-of-(2N+2)-columns=0}
\alpha_0
+
\alpha_1
+
\cdots
+
\alpha_{a}
+
\beta_0
+
\beta_1
+\cdots
+\beta_{a}
=
\mathbf{0};
\end{equation}

\smallskip\item[{\bf (ii)}]
for every $\nu=0\cdots a$, there holds the rank inequality: 
\begin{equation}
\label{(ii) of M^n_2c}
\rank_{\mathbb{K}}\,
\big\{
\alpha_0,\dots,\widehat{\alpha_{\nu}},\dots,\alpha_{a},\alpha_{\nu}+(\beta_0+\beta_1+\cdots+\beta_{a})
\big\}
\leqslant 
a-1;
\end{equation}

\smallskip\item[{\bf (iii)}] 
for every $\tau=0\cdots a-1$, for every $\rho=\tau+1 \cdots a$, there holds:
{\footnotesize
\begin{equation}
\label{(iii) of M^n_2c}
\rank_{\mathbb{K}}\,
\big\{
\alpha_0+\beta_0,\alpha_1+\beta_1,\dots,\alpha_\tau+\beta_\tau,
\alpha_{\tau+1},\dots,\widehat{\alpha_\rho},\dots,\alpha_{a},
\alpha_\rho+(\beta_{\tau+1}+\cdots+\beta_{a})
\big\}
\leqslant 
a-1.
\end{equation}
}
\end{itemize}

\medskip
From now on, we only consider the closed points
in each scheme. For instance, we shall regard:
\[
\mathbb{P}(\Omega_{\mathbf{P}})\,
=\,
\Big\{
(z, [\xi])\
\colon\
\forall\,
z\in\mathbf{P},\,
\xi\in
\mathrm{T}_{\mathbf{P}}\big{|}_z
\Big\}.
\]
By the same reasoning as in~\cite[Proposition 9.3]{Xie-2015-arxiv}, we get:
\begin{Proposition}
\label{characterization-BS-for-nonvanishing-coordinates}
For generic choices
of parameters $A_{\bullet}^{\bullet}, M_{\bullet}^{\bullet;\bullet}$,
a point: 
\[
(z,[\xi])\
\in\
\mathbb{P}(\Omega_{\mathbf{P}})
\setminus
\{\zeta_0\cdots \zeta_N\neq 0\}
\]
lies in 
$
\mathsf{BS}
$
if and only if:
\[
\explain{recall~\thetag{\ref{M=(...)}}}
\qquad
\mathsf{M}\,
(z,\xi)
\in
\mathscr{M}_{2c+r}^N.
\eqno
\qed
\]
\end{Proposition}

Now, we introduce the engine of MCM
(slightly different from the original~\cite[Lemma 9.5]{Xie-2015-arxiv}):

\medskip
\noindent
{\bf Core Lemma.}
{\em \
For all positive integers $2\leqslant a \leqslant b$, there hold the codimension estimates:}
\[
\cdim\,
\mathscr{M}_{b}^a\,
\geqslant\,
a
+
b
-
1.
\eqno
\qed
\]

\smallskip
\noindent
{\em `Naive proof'.}\,
First of all, the equation~\thetag{\ref{sum-of-(2N+2)-columns=0}}
eliminates
the first variable column:
\[
\alpha_0\,
=\,
-\,
(
\alpha_1
+
\cdots
+
\alpha_{a}
+
\beta_0
+
\beta_1
+\cdots
+\beta_{a}
),
\]
so it contributes codimension value $b$.
Next, 
denoting
$
S_i
:=
\sum_{j=i}^a\,
\beta_j$
for
$i=0\cdots a$,
we may rewrite the restriction~\thetag{\ref{(ii) of M^n_2c}}
as:
\[
\rank_{\mathbb{K}}\,
\big\{
\alpha_0,\dots,\widehat{\alpha_{\nu}},\dots,\alpha_{a},\alpha_{\nu}+S_0
\big\}
\leqslant 
a-1
\qquad
{\scriptstyle (\nu\,=\,0\,\cdots\, a)}.
\]
By~\thetag{\ref{sum-of-(2N+2)-columns=0}},
the sum of all columns above vanishes, hence we 
can drop the first column and state it equivalently as:
\begin{equation}
\label{(ii)' of M^n_2c}
\rank_{\mathbb{K}}\,
\big\{
\alpha_1,\dots,\widehat{\alpha_{\nu}},\dots,\alpha_{a},\alpha_{\nu}+S_0
\big\}
\leqslant 
a-1
\qquad
{\scriptstyle (\nu\,=\,0\,\cdots\, a)}.
\end{equation}
Similarly, we can reformulate~\thetag{\ref{(iii) of M^n_2c}} equivalently as:
\begin{equation}
\label{(iii)' of M^n_2c}
\aligned
\rank_{\mathbb{K}}\,
\big\{
\alpha_1+S_1-S_2,\dots,\alpha_\tau+
S_{\tau}-S_{\tau+1},
&
\alpha_{\tau+1},\dots,\widehat{\alpha_\rho},\dots,\alpha_{a},
\alpha_\rho+S_{\tau+1}
\big\}
\leqslant 
a-1
\\
&
{\scriptstyle (\tau\,=\,0\,\cdots\, a-1,\,\,
\rho\,=\,\tau+1\,\cdots\, a)}.
\endaligned
\end{equation}
Observe in~\thetag{\ref{(ii)' of M^n_2c}},
\thetag{\ref{(iii)' of M^n_2c}} that
the variable columns
$S_0,\dots,S_{a}$ have distinct status,
and moreover that for $i=1\cdots a$, subsequently, each variable $S_i$ satisfies nontrivial new equations
involving only the former variables $\alpha_{\bullet}, S_1,\dots, S_{i-1}$.
Thus, the restrictions~\thetag{\ref{(ii)' of M^n_2c}},
\thetag{\ref{(iii)' of M^n_2c}} shoud contribute at least $a+1$ codimension value. Summarizing, we should have:
\[
\cdim\,
\mathscr{M}_{b}^a\,
\geqslant\,
b
+
(
a
+
1
)
\,
\geqslant\,
a+b-1.
\eqno
\qed
\]

\begin{Remark}
However, a rigorous proof
(cf.~\cite[Subsection~10.6]{Xie-2015-arxiv})
 is much more
demanding and delicate, because of the unexpected algebraic complexity behind (cf.~\cite[Subsection~10.7]{Xie-2015-arxiv}).
\end{Remark}

Thereby, 
we can exclude positive-dimensional base locus in Proposition~\ref{characterization-BS-for-nonvanishing-coordinates}.

\begin{Proposition}
\label{why the core lemma is essential}
For generic choices
of parameters $A_{\bullet}^{\bullet}, M_{\bullet}^{\bullet;\bullet}$,
the base locus over the `coordinates nonvanishing part':
\[
\mathsf{BS}
\setminus
\{\zeta_0\cdots \zeta_N\neq 0\}
\]
is discrete or empty.
\end{Proposition}

\begin{proof}
The proof goes much the same as that of Proposition~\ref{Bertini-type assertion},
in which the underlined codimension estimate is replaced by:
\[
\aligned
\cdim\,\mathscr{M}_{2c+r}^N\,
&
\geqslant\,
N+2c+r-1
\qquad
\explain{by the Core Lemma}
\\
\explain{use $2c+r\geqslant N$}
\qquad
&
\geqslant\,
N+N-1
\\
\explain{exercise}
\qquad
&
=\,
\dim\,
\mathbb{P}(\Omega_{\mathbf{P}})
\\
\explain{\smiley}
\qquad
&
=\,
\dim\,
\big(
\mathbb{P}(\Omega_{\mathbf{P}})
\setminus
\{\zeta_0\cdots \zeta_N\neq 0\}
\big).
\endaligned
\]
For the remaining details,
we refer the reader to~\cite[Propositions 9.6, 9.7]{Xie-2015-arxiv}.
\end{proof}

This is exactly the first emptiness assertion on the base loci in Subsection~\ref{subsection: Controlling the base loci}.
By much the same reasoning, we can also establish
the second one there 
(cf.~\cite[Proposition 9.11]{Xie-2015-arxiv}).


\end{document}